\documentclass[12pt]{amsart}
\usepackage{amsfonts,amssymb,amsmath}

\usepackage{mathrsfs}
\usepackage{bbm}

\usepackage{tikz}

\usepackage{xcolor}

\newcommand {\SC} {{\mathbb C}}

\newcommand {\SK} {{\mathbb K}}
\newcommand {\SN} {{\mathbb N}}
\newcommand {\SR} {{\mathbb R}}
\newcommand {\ST} {{\mathbb T}}

\newcommand {\SX} {{\mathbb X}}

\newcommand {\SZ} {{\mathbb Z}}

\let\oldphi=\phi

\renewcommand {\phi} {{\varphi}}
\newcommand {\tphi} {{\tilde\varphi}}

\newcommand {\al} {{\alpha}}

\newcommand {\dt} {{\delta}}
\newcommand {\Dt} {{\Delta}}

\newcommand {\e} {{\varepsilon}}
\newcommand {\bfe} {{\boldsymbol \e}}

\newcommand {\la} {{\lambda}}

\newcommand{\Ga}{{\Gamma}}

\newcommand {\phin} {{\varphi_{i_{n+1}}}}

\newcommand {\tw} {\widetilde{w}}

\newcommand {\bone} {{\bf 1}}

\newcommand {\cD} {{\mathcal D}}

\newcommand {\cI} {{\mathcal I}}

\newcommand {\cT} {{\mathcal T}}

\newcommand {\tG} {{\widetilde G}}

\newcommand {\tdt} {{\widetilde\dt}}

\newcommand {\hg} {{\hat g}}

\newcommand {\bPhi} {{\bar\Phi}}

\newcommand {\Ds} {\displaystyle}
\newcommand {\Ts} {\textstyle}

\newcommand{\Log}{{\operatorname{Log }}}

\def\supp{\mathop{\rm supp}}
\def\dist{\mathop{\rm dist}}

\def\lsgn{\mathop{\overline{\mbox{\rm sign}}}}

\newcommand {\Scirc} {\raise.2ex\hbox{$\scriptstyle\circ$}}

\newcommand {\mand} {{\quad\mbox{and}\quad}}

\renewcommand {\mid} {{\,\,\,\colon\,\,\,}}

\renewcommand {\mid} {{\,\,\,\colon\,\,\,}}

\newcommand{\sline}{{\smallskip

\noindent}}

\newcommand{\Ba}[1]{\begin{array}{#1}}
\newcommand{\Ea}{\end{array}}
\newcommand{\Be}{\begin{equation}}
\newcommand{\Ee}{\end{equation}}
\newcommand{\Bea}{\begin{eqnarray}}
\newcommand{\Eea}{\end{eqnarray}}
\newcommand{\Beas}{\begin{eqnarray*}}
\newcommand{\Eeas}{\end{eqnarray*}}
\newcommand{\Benu}{\begin{enumerate}}
\newcommand{\Eenu}{\end{enumerate}}
\newcommand{\Bi}{\begin{itemize}}
\newcommand{\Ei}{\end{itemize}}

\newcommand{\BR}{\begin{Remark} \em}
\newcommand{\ER}{\end{Remark}}

\newcommand{\BE}{\begin{example} \em}
\newcommand{\EE}{\end{example}}

\newcounter{remark}

\newtheorem{theorem}[equation]{T{\hskip 0pt\footnotesize\bf HEOREM}}
\newtheorem{proposition}[equation]{P{\hskip 0pt\footnotesize\bf ROPOSITION}}
\newtheorem{corollary}[equation]{C{\hskip 0pt\footnotesize\bf OROLLARY}}
\newtheorem{lemma}[equation]{L{\hskip 0pt\footnotesize\bf EMMA}}
\newtheorem{Remark}[equation]{R{\hskip 0pt\footnotesize\bf EMARK}}
\newtheorem{definition}[equation]{D{\hskip 0pt\footnotesize\bf EFINITION}}
\newtheorem{example}[equation]{E{\hskip 0pt\footnotesize\bf XAMPLE}}

\newcommand {\ProofEnd} {
             \begin{flushright} \vskip -0.2in $\Box$ \end{flushright}}

\newcommand {\Au}{{\texttt{A2}}}
\newcommand {\At}{{\texttt{A3}}}
\newcommand {\D}{{\texttt{D}}}
\def\bbone{{\mathbbm 1}}

\newcommand {\A} {{\mathscr{A}}}
\newcommand {\G} {{\mathscr{G}}}

\DeclareFontEncoding{FMS}{}{}
\DeclareFontSubstitution{FMS}{futm}{m}{n}
\DeclareFontEncoding{FMX}{}{}
\DeclareFontSubstitution{FMX}{futm}{m}{n}
\DeclareSymbolFont{fouriersymbols}{FMS}{futm}{m}{n}
\DeclareSymbolFont{fourierlargesymbols}{FMX}{futm}{m}{n}
\DeclareMathDelimiter{\VERT}{\mathord}{fouriersymbols}{152}{fourierlargesymbols}{147}

\def \<{\langle}
\def\>{\rangle}

\def \sign{\operatorname{sign}}

\definecolor{ggcol}{cmyk}{.74, 0, 1, .41} 

\begin{document}

\title{The WCGA in $L^p(\log L)^\al$ spaces }

\author{G. Garrig\'os}
\address{Gustavo Garrig\'os
\\
Departamento de Matem\'aticas
\\
Universidad de Murcia
\\
30100 Murcia, Spain} \email{gustavo.garrigos@um.es}



\begin{abstract}
We present some new results concerning Lebesgue-type inequalities for the Weak Chebyshev Greedy Algorithm (WCGA)
in uniformly smooth Banach spaces $\SX$.
First, we generalize Temlyakov's theorem \cite{Tem14} to cover situations in which the modulus of smoothness and the $\At$ parameter
are not necessarily power functions. Secondly, we apply this new theorem to the Zygmund spaces $\SX=L^p(\log L)^\al$,
with $1<p<\infty$ and $\al\in\SR$, and show that, when the Haar system is used, then optimal recovery of $N$-sparse signals
occurs when the number of iterations is $\oldphi(N)=O(N^{\max\{1,2/p'\}}
\,(\log N)^{|\al| p'})$. Moreover, this quantity is sharp when $p\leq 2$.
Finally, an expression for $\oldphi(N)$ in the case of the trigonometric system is also given.
\end{abstract}

\maketitle

\section{Introduction}
\setcounter{equation}{0}

In this paper we consider several theoretical aspects regarding $N$-term approximation in a Banach space $(\SX,\|\cdot\|)$,
over a field $\SK=\SR$ or $\SC$. 

A fundamental question in this topic is, given a dictionary $\cD=\{\phi_i\}_{i\in\cI}$ in $\SX$, and the corresponding
set of $N$-sparse vectors
\[
\Sigma_N=\Sigma_N(\cD):=\Big\{\sum_{j=1}^Nc_j\phi_{i_j}\mid c_j\in\SK, \;\phi_{i_j}\in\cD\Big\},
\] 
then find constructive procedures
(algorithms) $\A_N:\SX\to \Sigma_N$, where for all $f\in\SX$ the quantity $\|f-\A_N(f)\|$ is as close as possible to the \emph{best 
error of $N$-term approximation}, defined by
\[
\sigma_N(f,\cD):=\dist(f,\Sigma_N)=\inf\Big\{\|f-g\|\mid g\in \Sigma_N(\cD)\Big\}.
\]
Once an algorithm $\A_N$ is fixed, one can quantify the above statement by considering the associated \emph{Lebesgue-type inequality},
which amounts to find the smallest value of $\oldphi(N)$ so that
\Be
\|f-\A_{\oldphi(N)}(f)\|\,\leq\,C\,\sigma_N(f),\quad \forall\;f\in\SX,
\label{AN}
\Ee
with $C$ a fixed universal constant (if it exists). Observe, in particular, that \eqref{AN} guarantees optimal recovery 
of all $N$-sparse signals after $\oldphi(N)$ iterations, that is 
\[
\A_{\oldphi(N)}(f)=f,\quad\forall \;f\in\Sigma_N(\cD).
\]
Ideally, one would like to find algorithms $\A_N$ so that \eqref{AN} holds with $\oldphi(N)=N$ (and $C=1$).
But this is hardly possible in many situations (a notable exception being when $\cD$ is an orthonormal basis in a Hilbert space).
For instance, in the classical case when $\cD$ is the trigonometric system in $L^p(\ST)$, $p\not=2$,
it is still a relevant open question to find one such (constructive) algorithm.

\

In this paper we shall be interested in the \emph{Weak Chebyshev Greedy Algorithm} (WCGA), which
was introduced by Temlyakov in \cite{Tem01} as a  generalization to Banach spaces of the
celebrated \emph{Orthogonal Matching Pursuit} (OMP) from Hilbert spaces. 
We refer to \cite{Tem11, Tem15, Tem18}, and references therein, for background on this topic.

\

Lebesgue-type inequalities for the WCGA were proved in \cite{LivTem14,Tem14}; see also \cite[Chapter 8]{Tem18} for a historical overview.
One the features of WCGA is that it has good approximation properties for the trigonometric system in $L^p$.
Indeed, it was shown in \cite[(4.3)]{Tem14} that, if $p>2$, then Lebesgue inequalities hold with only $\oldphi(N)=O(N\log N)$ iterations.
This seems to be the best known result with a constructive algorithm in that setting.
Likewise, for the univariate Haar system in $L^p$, if $1<p\leq2$, then it suffices with $\oldphi(N)=O(N)$ iterations; see \cite[(4.7)]{Tem14}.

\

The above results are special cases a deep theorem proved by Temlyakov in \cite[Theorem 2.8]{Tem14},
which we describe in detail below. In that theorem, the number of iterations $\oldphi(N)$ is estimated in terms 
of some intrinsic properties of the pair $(\SX,\cD)$, namely, 
the power type of the modulus of smoothness of $(\SX,\|\cdot\|)$, and the power function associated with the so-called 
property $\At$ of $\cD$; see \eqref{phiTem} below.

\

Our main result in this paper, Theorem \ref{th_newG}, will be a generalization of Temlyakov's theorem,
which allows to cover situations in which the modulus of smoothness and the $\At$ parameters
are not necessarily power functions. This is actually needed in some special cases, such as when $\SX=L^p(\log L)^\al$,
for which additional log factors appear naturally. Our next results, Theorems \ref{th_Lpal} and \ref{th3}, will 
be applications of Theorem \ref{th_newG} to this setting, for two special dictionaries, the Haar and the trigonometric system.

\

We next give a more detailed description of these results.

\subsection{Statements of results}

We assume that $(\SX,\|\cdot\|)$ is a uniformly smooth Banach space, meaning that 
its modulus of smoothness
\Be
\rho_\SX(t):=\sup_{\|f\|=\|g\|=1} \,\tfrac12\,\Big(\|f+tg\|+\|f-tg\|-2\|f\|\Big),\quad t\in\SR.
\label{rho}
\Ee
satisfies $\rho_\SX(t)=o(t)$. Given $f\in\SX$ with $f\not=0$, we let $F_f\in \SX^*$ be the associated norming functional, that is,
the (unique) element in $\SX^*$ such that
\Be
\|F_f\|_{\SX^*}=1,\mand F_f(f)=\|f\|.
\label{F_f}
\Ee
Uniqueness follows from the smoothness of the norm $\|\cdot\|$.

We say that $\cD=\{\phi_i\}_{i\in\cI}$ is a \emph{dictionary} in $\SX$, if it consists of non-null vectors whose 
closed linear span is $\SX$, that is
\[
\big[\phi_i\big]_{i\in\cI}=\SX.
\]
We do not assume the dictionary elements to be normalized, although 
as a consequence of later properties $\cD$ will be \emph{semi-normalized}, that is
\[
\frak{c_0}\leq \|\phi_i\|\leq \frak{c_1}, \quad \forall\,\phi_i\in\cD,
\]
for some constants $\frak{c_1}\geq \frak{c_0}>0$; see \S\ref{R_sn} below.

\begin{definition}{\bf Weak Chebyshev Greedy Algorithm (WCGA).}\label{wcga}
Given a fixed $\tau\in(0,1]$, a $\tau$-WCGA associated with $(\SX,\|\cdot\|,\cD)$ is 
any collection of mappings 
\[
\G_N:\SX\to\Sigma_N(\cD),\quad N=1,2,\ldots
\]
with the following properties: 

\sline  Given $f\in \SX\setminus\{0\}$, we let $f_0:=f$ and define inductively vectors $\phi_{i_1},\ldots, \phi_{i_n}$ in $\cD$ and $f_1,\ldots, f_n\in\SX$ by the following procedure:
at step $n+1$ we pick any  $\phin\in\cD$ such that
\Be
|F_{f_n}(\phin)|\geq \,\tau\,\sup_{\phi\in\cD}|F_{f_n}(\phi)|,
\label{FfphiD}
\Ee
and let $\G_{n+1}(f)$ be any element in $[\phi_{i_1},\ldots,\phi_{i_{n+1}}]$ such that
\Be
\|f-\G_{n+1}(f)\|=\dist\big(f,[\phi_{i_1},\ldots,\phi_{i_{n+1}}]\big).
\label{cheby}
\Ee
Then we set $f_{n+1}=f-\G_{n+1}(f)$, and iterate the process (indefinitely, or until the remainder $f_{n+1}=0$). 

\noindent If at some stage we have $f_n=0$, then we just let $\G_{n+k}(f)=\G_n(f)=f$ for all $k\geq1$.
\end{definition}

\BR Note that such algorithms can always be constructed when $\tau<1$, and for some dictionaries also when $\tau=1$ (namely, 
when the sup in \eqref{FfphiD} is attained within $\cD$).
\ER

We next define the three key properties that are needed to prove Lebesgue-type inequalities for WCGA. 
The first one is a generalization of a property given in \cite[Definition 1.13]{DGHKT21}.

\begin{definition}
Let $Q(t)$ be a positive increasing function for $t\in(0,\infty)$, with $Q(0)=0$.
We say that $(\SX,\|\cdot\|,\cD)$ satisfies {\bf property $\D(Q)$} if
\Be
\dist(f,[\phi])\leq \|f\|\,\Big(1-Q(|F_f(\phi)|)\Big),\quad \forall\,\phi\in\cD,\,f\in\SX\setminus\{0\}.
\label{DQ}
\Ee
\end{definition}

The next definition coincides with property $\Au$ from \cite{LivTem14, Tem14}.

\begin{definition}
Let $N<D$ be positive integers and $k_N>0$.
We say that $\Sigma_N(\cD)\in\Au(k_N, D)$ if
\Be
\label{kND}
\|\sum_{j\in A}a_j\phi_j\|\leq k_N\,\|\sum_{j\in B}a_j\phi_j\|, \quad \forall\,a_j\in\SK, \;\forall\,A\subset B\mid |A|\leq N,\;|B|< D
\Ee
If the above holds for all $D<\infty$, we just write $\Sigma_N(\cD)\in\Au(k_N)$.
\end{definition}

Our third definition is a slight generalization of  property $\At$ from \cite{Tem14}.

\begin{definition}
Let $N<D$ be positive integers and let $\{H(k)\}_{k=1}^\infty$ be an increasing sequence of positive numbers. 
We say
 $\Sigma_N(\cD)\in\At(H,D)$ if
\Be
\label{AtH}
\sum_{j\in A}|a_j|\leq H(|A|)\,\|\sum_{j\in B}a_j\phi_j\|, \quad \forall\,a_j\in\SK, \;\forall\,A\subset B\mid |A|\leq N,\;|B|< D.
\Ee
If the above holds for all $D<\infty$, we just write $\Sigma_N(\cD)\in\At(H)$.
\end{definition}

\

Finally, we recall that a positive sequence  $\{G(k)\}_{k=1}^\infty$ is called \emph{1-quasi-convex} 
if
\Be
\frac{G(k)}{k}\,\leq \,
\frac{G(k+1)}{k+1},\quad \forall\, k\in\SN.
\label{qc}
\Ee 
As an example, if  $G(t)$ is a positive convex function in $(0,\infty)$ with $G(0^+)=0$, 
then $\{G(k)\}_{k=1}^\infty$ is 1-quasi-convex.
This is the case, for instance, for the functions 
\Be
G(t)=t^p\,\big(\log(c+t)\big)^\al,
\label{Gqc}
\Ee
if $p=1$ and $\al\geq0$, 
or if $p>1$ and $\al\in\SR$ (for a sufficiently large $c\geq e$).

\

The precise statement of our main result is now the following.

\begin{theorem}\label{th_newG}
Let $(\SX,\|\cdot\|)$ be a Banach space, $\cD$ a dictionary, $\tau\in(0,1]$ and $\G_n:\SX\to\Sigma_n$ a  $\tau$-WCGA.
Let $D>N\geq1$ be fixed. Let $k_N>0$ and let $Q(t)$, $H(n)$ be positive and increasing  
functions such that the following properties hold
\Benu
\item[(i)] $(\SX,\|\cdot\|,\cD)$ satisfies $\D(Q)$

\item[(ii)] $\Sigma_N$ satisfies property $\Au(k_N,D)$.

\item[(iii)]  $\Sigma_N$ satisfies property $\At(H,D)$.
\Eenu
Let $\la_1>1$. Assume further that the sequence
\Be
G(n)= \Big[Q\Big(\frac{c(\tau)}{H(n)}\Big)\Big]^{-1},  \quad \mbox{with}\quad 
c(\tau)\,=\,\tfrac\tau2(1-\tfrac1{\sqrt{\la_1}})\,.
\label{Gn}
\Ee
is 1-quasi-convex. If we let
\Be
\oldphi(N)\,:=\,8\,\ln\Big[\frac{8(1+\la_1)k_N}{\sqrt{\la_1}-1}\Big]\,G(2N),
\label{phiHG}
\Ee
 then it holds
\Be
\Big\|x-\G_{\oldphi(N)}(x)\Big\|\leq \la_1\,\|x-\Phi\|,\quad \forall\,x\in\SX, \;\Phi\in\Sigma_N,
\label{xNG}
\Ee
provided that $N+\oldphi(N)< D$.
\end{theorem}

\

We now make some comments about this theorem.


\sline a) The result obtained by Temlyakov in \cite[Theorem 2.8]{Tem14} corresponds to 
the case when $\la_1$ is a (possibly large) universal constant, and 
\[
H(N)=V_N\,N^r \mand Q(t)=c\,t^{q'}, 
\]
where $q>1$ is the power type of the modulus of smoothness, ie 
$\rho_\SX(t)=O(t^q)$. In that case, the required number of iterations becomes
\Be
\oldphi(N)\,=\, C_1\,(V_N/\tau)^{q'}\,\log(1+k_N)\,N^{rq'},
\label{phiTem}
\Ee
for some $C_1>0$, provided  that $rq'\geq1$. Our contribution gives an additional explicit form
for the constants when the parameter $\la_1$ approaches 1. 

\sline b) As we show in Proposition \ref{P_dt} below, if $\cD$ is normalized, then condition $\D(Q)$ always holds with 
\[
Q(t)=2\dt_{\SX^*}(t/2),
\]
where $\dt_{\SX^*}(t)$ is the \emph{modulus of convexity} of the dual space $\SX^*$. This is also a new result.
In many practical cases the asymptotic behavior of $\dt_{\SX^*}(t)$ is well-known, so one can use property $\D(Q)$ with no need
to compute $\rho_\SX(t)$. 

\sline c) As was discussed in \cite[Remark 2.10]{DGHKT21}, in some special cases it is possible to prove that 
$(\SX,\|\cdot\|,\cD)$ satisfies property
$\D(Q)$ with a function $Q(t)$ which is considerably better than $\dt_{\SX^*}(t)$ (for $t$ near 0). 
For instance, if $\SX=\ell^p$ and $\cD$ is the canonical basis,
then one can take $Q(t)=c_pt^{p'}$, which gives better results than $\dt_{\SX^*}(t)=O(t^{\max\{p',2\}})$ when $p>2$.
Other examples (with power type) were given in \cite[Proposition 4.12 and Lemma 5.7]{DGHKT21}.

\sline d) The assumption that $G(n)$ in \eqref{Gn} is 1-quasi-convex is only made for convenience. 
Alternatively, one could replace $G(n)$ by any convex majorant (hence, 1-quasi-convex). In practice, 
quasi-convexity is easily verified after substituting the functions $Q(t)$ and $H(n)$ into \eqref{Gn}; see the example in 
\eqref{Gqc}.

\sline e) As in \cite{Tem14}, the conclusion \eqref{xNG} in the previous theorem also holds when the assumptions $\Au$ and $\At$
are required \emph{only} on the individual sparse element $\Phi=\sum_{j\in T}x_j\phi_j$, with $|T|\leq N$ (and not necessarily in 
all $\Phi\in\Sigma_N$).
Namely, in this case the requierement would be that \eqref{kND} and \eqref{AtH} must hold for all $A\subset T$ and 
all scalars $a_j\in\SK$ such that $a_j=x_j$, $j\in A$.

\

Our second result is an application of Theorem \ref{th_newG} to the case when $\SX=L^p(\log L)^\al$;
see \S\ref{S_Lpal} below for the precise definition. We stress that, when $1<p\leq 2$, the number of iterations
which are derived from the above theorem, namely
\[
\oldphi(N)=O\Big(N\,\big(\log(e+N)\big)^{p'|\al|}\Big),
\]
is actually (asymptotically) optimal for all $\al\in\SR$.

\begin{theorem}\label{th_Lpal}
Let $1<p<\infty$ and $\al\in\SR$, and let $\SX=L^p(\log L)^\al$ be as in \S\ref{S_Lpal}.
Let $\cD$ be the (normalized) Haar basis in $\SX$.
Then 

\sline a) there exists a constant $C>1$ such that the WCGA satisfies
\Be
\Big\|f-\G_{\oldphi(N)}(f)\Big\|\leq \,2\,\sigma_N(f) 
,\quad \forall\,f\in\SX,\;N\in\SN,
\label{foldphiN}
\Ee
where
\Be
\label{phiLpal}
\oldphi(N)=\begin{cases}
C\,N^\frac2{p'}\,\big(\log(e+N)\big)^{2\al_+} &   \mbox{when $p>2$}\\
C\,N\,\big(\log(e+N)\big)^{p'|\al|} &   \mbox{when $1<p\leq2$}.
\end{cases}
\Ee
\sline b) if for some sequence $\psi(N)$ the WCGA satisfies
\Be
\Big\|f-\G_{\psi(N)}(f)\Big\|\leq \,C\,\sigma_N(f) 
,\quad \forall\,f\in\SX,\;N\in\SN,
\label{fpsiN}
\Ee
then necessarily $\psi(N)\geq c'\, N\,\big(\log(e+N)\big)^{|\al|p'}$, for some $c'>0$.
\end{theorem}

\BR
We remark that, when $p>2$, it is an open question already for $L^p$ spaces (case $\al=0$) whether 
$\oldphi(N)\approx N^{2/p'}$ iterations are necessary to ensure \eqref{foldphiN}; 
see \cite[Open Problem 8.3, p. 448]{Tem18}. 
\ER

\

Finally, in \S\ref{S_trig} we give a similar application in the case that $\SX=L^p(\log L)^\al$ and 
 $\cD=\{e^{inx}\}_{n\in\SZ}$ is the trigonometric system. See Theorem \ref{th3} below for details.

\section{Preliminaries}
\setcounter{equation}{0}

\subsection{\bf About seminormalization of $\cD$}\label{R_sn}
We claim that the two properties $\D(Q)$ and $\At(H,D)$ imply that the dictionary $\cD$ must be semi-normalized.
Indeed, if $\Sigma_1$ satisfies $\At(H,D)$ then
\Be
\label{sn1}
\|\phi\|\geq 1/H(1), \quad \forall\, \phi\in\cD.
\Ee
On the other hand, $\D(Q)$ implies that $Q\big(|F_f(\phi)|\big)\leq 1$ for all $\phi\in\cD$ and $f\in\SX\setminus\{0\}$.
Setting $f=\phi$ and using $F_\phi(\phi)=\|\phi\|$, this gives 
\Be
\label{sn2}
\|\phi\|\leq Q^{-1}(1), \quad \forall\, \phi\in\cD.
\Ee
Conversely, suppose that $\cD=\{\phi_j\}$ is a dictionary satisfying any of the properties
$\D(Q)$, $\Au(k_N,D)$ or $\At(H,D)$, and let $\tphi_j={\la_j}\phi_j$ for scalars $\la_j$ such that
\[
0<\frak{c_0}\leq |\la_j|\leq \frak{c_1}, \quad \forall\,j.
\]
It is then easily seen that the new dictionary $\widetilde\cD=\{\tphi_j\}$ satisfies the corresponding properties 
with new parameters, namely
\[
\mbox{$\D\big(Q(\cdot/\frak{c_1})\big)$, \quad $\Au(k_N,D)$ \quad or \quad  $\At(H/\frak{c_0},D)$}
\]
We also remark that if $\G_N$ is $\tau$-WCGA for $\cD$, then it is also a $(\tau\frak{c_0}/\frak{c_1})$-WCGA for $\widetilde{\cD}$.

\subsection{About condition $\D(Q)$}\label{S_DQ}

%

We give a practical criterion which ensures that condition $\D(Q)$ holds. 
Let $(\SX,\|\cdot\|)$ be a Banach space with \emph{modulus of smoothness}
\[
\rho_{\SX}(t)=\sup_{\|x\|=\|y\|=1}\tfrac12\Big(\|x+ty\|+\|x-ty\|-2\|x\|\Big), \quad t\in\SR.
\]
We denote by $\dt(s)=\dt_{\SX}(s)$ its \emph{modulus of convexity}, that is
\[
\dt_{\SX}(s)=\inf_{{\|x\|=\|y\|=1}\atop{\|x-y\|=s}}\Big(\frac{\|x\|+\|y\|}2-\Big\|\frac{x+y}2\Big\|\Big),  \quad s\in[0,2].
\]
Next, we consider the following related function, introduced by  Figiel \cite{Fi76},
\Be
\label{tdt}
\tdt_{\SX^*}(s):=\sup_{t\geq0}\Big(\tfrac12st-\rho_{\SX}(t)\Big), \quad s\geq0.
\Ee
Assume for simplicity that $\SX$ is uniformly smooth, that is $\rho_\SX(t)=o(t)$ when $t\to0$ (so in particular, $\SX$ is reflexive).
Then, it is easily seen that $Q(s)=\tdt_{\SX^*}(s)$ is a convex increasing function with $Q(0)=0$.
Moreover, it is shown in \cite[Proposition 1]{Fi76} (see also \cite[Proposition 1.e.6]{LZ}) that $\tdt_{\SX^*}(s)$ is
 ``\emph{equivalent}'' to 
$\dt_{\SX^*}(s)$ (for small $s$), in the sense that
\[
\dt_{\SX^*}(s/2) \leq \tdt_{\SX^*}(s)\leq \dt_{\SX^*}(s), \quad s\in[0,2].
\]
Also, $\tdt_{\SX^*}(s)$ is the greatest convex minorant of $\dt_{\SX^*}(s)$. In particular, $\tdt_{\SX^*}=\dt_{\SX^*}$
when the later is a convex function. In many examples of Banach spaces $\SX$, the behavior of the function $\dt_{\SX^*}(s)$ 
is well-known (sometimes quite explicitly). For instance, if $\SX=L^p$, $1<p<\infty$, then
\[
\dt_{L^{p'}}(s) =c_q \,s^{q}+ o(s^q),\quad \mbox{with $q=\max\{2,p'\}$;}
\]
see \cite[p.63]{LZ}.
Our main result in this section is the following.

\begin{proposition}\label{P_dt}
If $(\SX,\|\cdot\|)$ is uniformly smooth, then every \emph{normalized} dictionary $\cD$ in $\SX$ satisfies property 
$\D(Q)$ with $Q(s)=2\tdt_{\SX^*}(s)$.
\end{proposition}
\begin{proof}
It suffices to prove \eqref{DQ} for $f=x\in\SX$ with $\|x\|=1$. Let $F_x\in\SX^*$ be the norming functional of $\SX$, and 
given $\phi\in\cD$, let $\nu=\overline{\mbox{sign}}\,F_x(\phi)$. Then, for every $t\geq0$, using \cite[Proposition 2.1]{DGHKT21},
we have
\[
\dist(x,[\phi])\leq \|x-\nu t\phi\|\leq \|x\|-t\,|F_x(\phi)|\,+\,2\,\rho_{\SX}(t).
\]  
Taking the infimum over all $t\geq0$ we obtain
\[
\dist(x,[\phi])\leq 1-2\sup_{t\geq0}\Big(\tfrac12t\,|F_x(\phi)|\,-\,\rho_{\SX}(t)\Big)=1-2\tdt_{\SX^*}(|F_x(\phi)|).
\]
\end{proof}
\BR
If the dictionary is not normalized, but we assume that $0<\|\phi\|\leq \frak{c_1}$, for all $\phi\in\cD$, then
the previous result gives
\[
\dist(x,[\phi])\leq 1-2\tdt_{\SX^*}(|F_x(\phi/\|\phi\|)|)\leq 1-2\tdt_{\SX^*}(|F_x(\phi)|/\frak{c_1}).
\]
So property $\D(Q)$ holds with $Q(s)=2\tdt_{\SX^*}(s/\frak{c_1})$, which is also a function equivalent to $\dt_{\SX^*}(s)$.
\ER

\subsection{About condition $\At(H)$}\label{S_AtH}

In practice, it is quite common that $(\SX,\cD)$ satisfies properties $\Au$ or $\At$ with depth $D=\infty$.
Our first observation is that this implies that $\cD$ has a biorthogonal dual system.

\begin{lemma}
Let $\cD=\{\phi_j\}_{j=1}^\infty$ be a dictionary in $\SX$. Assume that one of the following properties hold
\Benu
\item[(i)] There exists $k_1>0$ such that $\Sigma_1(\cD)\in \Au(k_1;D)$, for all $D<\infty$
\item[(ii)] There exists $H(1)>0$ such that $\Sigma_1(\cD)\in \At(H;D)$, for all $D<\infty$.
\Eenu
Then, there exists $\{\phi^*_j\}_{j=1}^\infty$ in $\SX^*$ such that 
$\{\phi_j,\phi^*_j\}_{j=1}^\infty$ is a biorthogonal system, ie
\[
\phi^*_j(\phi_i)=0,\;\;\mbox{if $j\not=i$},\mand \phi^*_j(\phi_j)=1.
\]
\end{lemma}
\begin{proof}
This is a consequence of \cite[Theorem 6.1, page 54]{Sin1}. Indeed, if (i) holds then
\[
\big\|\sum_{j=1}^na_j\phi_j\big\|\leq \sum_{j=1}^n\|a_j\phi_j\|\leq n\,k_1\,\big\|\sum_{i=1}^{n+m}a_i\phi_i\big\|,
\]
which implies biorthogonality by \cite[Theorem 6.1, ``$8^{\rm o}\Rightarrow2^{\rm o}$'']{Sin1}.
Similarly, if (ii) holds then
\[
\sum_{j=1}^n\frac{|a_j|}{2^j\,H(1)}\,\leq\,
\sum_{j=1}^n2^{-j}\,\big\|\sum_{i=1}^{n}a_i\phi_i\big\|\,\leq\, \big\|\sum_{i=1}^{n}a_i\phi_i\big\|,
\]
which implies biorthogonality by \cite[Theorem 6.1, ``$4^{\rm o}\Rightarrow2^{\rm o}$'']{Sin1}.
\end{proof}

So under this situation, the dictionary $\cD$ generates a dual  system $\cD^*$.
Then, a variation of \cite[Lemma 2.17]{DGHKT21} gives the following.

\begin{lemma}\label{L_A3H}
Let $\cD=\{\phi_j\}_{j=1}^\infty$ be a dictionary, with dual system $\cD^*=\{\phi^*_j\}_{j\geq1}$.
Then, $\Sigma_N\in\At(H,D)$, for all $N<D<\infty$, if we choose
\Be
\label{Hn}
H(n)=\sup_{|A|\leq n, |\e_j|=1}\,\big\|\sum_{j\in A} \e_j\phi^*_j\big\|_{\SX^*}.
\Ee
\end{lemma}
\begin{proof}
Take sets $A\subset B$, with $|A|\leq N$, and scalars $a_j\in\SK$. Let $\e_j=\lsgn a_j$, and denote 
\[
\bbone^*_{\bfe A}:=\sum_{j\in A} \e_j\phi^*_j\in\SX^*.
\] 
Then
\[
\sum_{n\in A}|a_n|=\bbone^*_{\bfe A}\Big(\Ts\sum_{n\in B} a_n\phi_n\Big)\leq \|\bbone^*_{\bfe A}\|_{\SX^*}
\,\big\|\sum_{n\in B} a_n\phi_n\big\|\leq H(N)\,\big\|\sum_{n\in B} a_n\phi_n\big\|.
\]
\end{proof}
In practice, the sequence $H(n)$ in \eqref{Hn} is equivalent to the fundamental function of $\cD^*$ in $\SX^*$,
which in many examples has an explicit expression.

\subsection{Quasi-convex sequences}

Given a positive sequence $w=\{w(j)\}_{j=1}^\infty$ we define its associated \emph{summing sequence} $\tw$ by
\Be
\tw(n):=\sum_{j=1}^n\frac{w(j)}j,\quad n\geq1.
\label{tw}
\Ee
\begin{lemma}\label{L1}
If $w=\{w(j)\}_{j=1}^\infty$ is non-decreasing then for all $N\in\SN$
\[
\sum_{j\,:\, 1\leq 2^j\leq N}w(2^j) < 2\tw(N).
\]
\end{lemma}
\begin{proof}
Let $\Dt_j=\{n\in\SN\mid 2^j\leq n<2^{j+1}\}$, which has cardinality $2^j$, $j=0,1,\ldots$
Then, if $J$ is the largest integer with $2^J\leq N$ we have
\[
\sum_{j=0}^J w(2^j)\leq \sum_{j=0}^J\frac{\sum_{n\in\Dt_j}w(n)}{2^j}< 2\sum_{j=0}^J\sum_{n\in\Dt_j}\frac{w(n)}{n}=2\tw(2^J)\leq 2\tw(N).
\]
\end{proof}

\begin{lemma}\label{L2}
If $w=\{w(j)\}_{j=1}^\infty$ is 1-quasi-convex then 

\sline a) $\tw(N)\leq w(N)$ for all $N\in\SN$

\sline b) $\tw$ is superadditive, that is,
\[
\tw(M+N)\geq \tw(M)+\tw(N),\quad \forall\,M,N\in\SN.
\]
\end{lemma}
\begin{proof}
The assertion a) follows from the definition of 1-quasi-convex, since
\[
\tw(N)=\sum_{n=1}^N\frac{w(n)}n\leq \sum_{n=1}^N\frac{w(N)}N =w(N).
\]
The assertion b) follows similarly from
\Beas
\tw(M+N) & = & \sum_{n=1}^N\frac{w(n)}n+\sum_{n=N+1}^{N+M}\frac{w(n)}n\\
& = &  \tw(N)+\sum_{j=1}^{M}\frac{w(j+N)}{j+N} \\
& \geq &  \tw(N)+\sum_{j=1}^{M}\frac{w(j)}{j}=\tw(N)+\tw(M).
\Eeas
\end{proof}

\section{The proof of Theorem \ref{th_newG}}
\setcounter{equation}{0}

In this section we give the proof of Theorem \ref{th_newG}.
We shall follow the main steps in the original proof of Temlyakov, see \cite[Theorem 2.8]{Tem14} or \cite[Theorem 8.7.18]{Tem18}, 
adapted to the new properties $\D(Q)$ and $\At(H,D)$.
For completeness, we give self-contained arguments of all the steps, although the main changes
will mostly appear in steps 1 and 4.

\subsection{Step 1. The iteration theorem}

The following result is a generalization of \cite[Theorem 3.1]{DGHKT21}, so we follow the notation presented there. 
Namely, if $f\in\SX\setminus\{0\}$, then
we write  
\[
f_n:=f-\G_n(f)\mand \Ga_n:=\supp\G_n(f),
\]
for the remainder and the supporting set of the $n$-th WCGA applied to $f$; see Definition \ref{wcga}.
Also, if $\Phi=\sum_{j\in T}a_j\phi_j\in\Sigma_N$ and $A\subset T$, then we denote 
\[
\Phi_A:=\sum_{j\in T\cap A}a_j\phi_j,\mand T_n:=T\setminus \Ga_n.
\]
We shall also make frequent use of \cite[Lemma 2.12]{DGHKT21}, which asserts that 
\Be
\label{null}
F_{f_n}(g)=0,\quad \forall\;g\in[\phi_j]_{j\in \Ga_n}.
\Ee

\begin{theorem}\label{th_iterG}
Let $D>N\geq1$. Assume that
\Benu
\item[(i)] $(\SX,\|\cdot\|,\cD)$ satisfies $\D(Q)$

\item[(ii)]  $\Sigma_N$ satisfies property $\At(H,D)$
\Eenu
Then, for every $f\in\SX\setminus\{0\}$, $\Phi=\sum_{j\in T}a_j\phi_j\in\Sigma_N$, and $\la>1$, and for all integers
$m,M\geq0$ such that $N+m+M<D$ the following holds
\Be
\|f_{m+M}\|\leq e^{-M/G(|A|)}\,\|f_m\|\,+\,\la\,\Big(\|f-\Phi\|+\|\Phi_B\|\Big),
\label{iterG}
\Ee
for all sets $A\subset T_k$ (with $A\not=\emptyset$), $B=T_k\setminus A$ and all $k\in[0,m)$, and where 
\Be
\label{G}
G(n)=\frac1{Q(c(\tau)/H(n))},  \quad \mbox{with}\quad 
c(\tau)\,=\,\tfrac\tau2\,(1-\tfrac1\la)\,.
\Ee
\end{theorem}
\begin{proof}
Given a fixed $n\in[m,m+M)$, condition $\D(Q)$ implies
\Be
\|f_{n+1}\|\leq \dist(f_n,[\phi_{i_{n+1}}])\leq \|f_n\|\,\Big(1-Q(|F_{f_n}(\phi_{i_{n+1}})|)\Big).
\label{fn1Q}
\Ee
By definition of the WCGA and \eqref{null}, for each (non-empty) $A\subset T$ we have
\[
\tau\,|F_{f_n}(\Phi_A)|=\tau\,|F_{f_n}(\Phi_{A\cap T_n})|\leq \Big(\sum_{A\cap T_n}|a_i|\Big)\,|F_{f_n}(\phi_{i_{n+1}})|.
\]
Now, the assumption $\Sigma_N\in \At(H,D)$  implies that
\[
\sum_{A\cap T_n}|a_i|\leq H(|A|)\,\|\Phi-\G_n(f)\|\leq H(|A|)\,\big(\|\Phi-f\|+\|f_n\|\big).
\]
In order to apply $\At$ we have used that $|A\cap T_n|\leq |T|\leq N$ and $|T\cup \Ga_n|\leq N+n\leq N+m+M<D$.
Thus, inserting these estimates into \eqref{fn1Q} we obtain
\[
\|f_{n+1}\|\leq  \|f_n\|\,\Big(1-Q\Big(\frac{\tau\,|F_{f_n}(\Phi_A)|}{H(|A|)(\|\Phi-f\|+\|f_n\|)}\Big)\Big),
\]
which is valid for all sets $A\subset T$.

Fix now an integer $k\in[0,m]$ and a set $A\subset T_k$, and let $B=T_k\setminus A$. Since $\Ga_k\subset\Ga_n$ we 
can use \eqref{null} to obtain
\Beas
|F_{f_n}(\Phi_A)| & = & |F_{f_n}(\Phi_A+\Phi_{\Ga_k\cap T}-\G_n(f))| = 
|F_{f_n}(\Phi_T-\Phi_B-f+f_n)| \\
& \geq & \|f_n\|-\|f-\Phi\|-\|\Phi_B\|.
\Eeas
So, we conclude that
\Beas
\|f_{n+1}\| &  \leq &   \|f_n\|\,\Big(1-Q\Big(\frac{\tau\,(\|f_n\|-\|f-\Phi\|-\|\Phi_B\|)_+}{H(|A|)(\|\Phi-f\|+\|f_n\|)}\Big)\Big).\\
\Eeas
Using in the denominator that $\|\Phi-f\|\leq \|f_n\|$ (when the numerator is not zero), this further simplifies into
\Beas
\|f_{n+1}\| & \leq &    \|f_n\|\,\Big(1-Q\Big(\frac{\tau\,(\|f_n\|-\|f-\Phi\|-\|\Phi_B\|)_+}{2\,H(|A|)\,\|f_n\|}\Big)\Big)\\
& = &    \|f_n\|\,\Big(1-Q\Big(\frac{\tau\,(1-u)_+}{2H(|A|)}\Big)\Big),
\Eeas
where we have let $u=(\|f-\Phi\|+\|\Phi_B\|)/\|f_n\|$.
Now, call
\[
\beta=Q\Bigg(\frac{(1-\tfrac1\la)\tau}{2H(|A|)}\Bigg)=\frac1{G(|A|)},
\]
and observe that, if $u\leq 1/\la$ then we have
\Be
\|f_{n+1}\| \leq \,(1-\beta)\, \|f_n\|.
\label{bet1}
\Ee
On the other hand, if $u\geq 1/\la$, by definition of $u$ we have
\[
\|f_n\|\leq \la\,\big(\|f-\Phi\|+\|\Phi_B\|\big),
\]
and therefore,
\Bea
\|f_{n+1}\| &  \leq & \|f_n\|\,= \,(1-\beta)\, \|f_n\|\,+\,\beta\,\|f_n\|\nonumber\\
& \leq & \,(1-\beta)\, \|f_n\|\,+\,\beta\,\la\,\big(\|f-\Phi\|+\|\Phi_B\|\big).\label{bet2}
\Eea
So, combining \eqref{bet1} and \eqref{bet2}, and calling  $v=\|f-\Phi\|+\|\Phi_B\|$ we obtain
\[
\|f_{n+1}\|-\la\,v\,\leq\, (1-\beta)\,(\|f_n\|-\la\,v).
\]
Since $\|f_{n+1}\|\leq \|f_n\|$ this implies 
\[
\big(\|f_{n+1}\|-\la\,v\big)_+\,\leq\, (1-\beta)\,(\|f_n\|-\la\,v)_+.
\]
We can now iterate for all $n\in[m,m+M)$ to obtain
\[
\|f_{m+M}\|-\la\,v\,\leq\, (1-\beta)^M\,(\|f_m\|-\la\,v)_+\,\leq \,(1-\beta)^M\,\|f_m\|.
\]
Finally, using the value of $v$ and $1-\beta\leq e^{-\beta}$ we obtain
\[
\|f_{m+M}\|\,\leq\, e^{-M\beta}\,\|f_m\|\,+\la\,(\|f-\Phi\|+\|\Phi_B\|).
\]
This corresponds exactly to \eqref{iterG}.
\end{proof}

\subsection{Step 2: Selection of sets $A_j$}\label{S4.2}

In the next step, we shall follow \cite[pp. 435--437]{Tem18}, and iteratively apply Theorem \ref{th_iterG}, 
with a suitably chosen selection of sets $A_j$, in order to obtain
the following result.
We have adapted the proof to include the new conditions $\D(Q)$ and $\At(H, D)$,
and have made more precise the value of the constants.

\begin{theorem}\label{th2G}
Let $(\SX,\|\cdot\|,\cD)$ and $1\leq N<D$ be such that 
\Benu
\item[(i)] $(\SX,\|\cdot\|,\cD)$ satisfies $\D(Q)$

\item[(ii)] $\Sigma_N$ satisfies property $\Au(k_N,D)$.

\item[(iii)]  $\Sigma_N$ satisfies property $\At(H,D)$.
\Eenu
Given $\la>1$ and $\dt>0$, there exists
$\beta_0=\beta_0(\la, \dt, k_N)>0$ such that, if $\beta\geq \beta_0$ and $\Phi=\sum_{j\in T}a_j\phi_j\in \Sigma_N\setminus\{0\}$, then
there exist positive integers $L, m_L\in\SN$ such that
\Be
2^{L-2} < |T|\mand m_L\leq \beta\,\sum_{j=1}^L G(2^{j-1}),
\label{LmLG}
\Ee
(with $G(n)$ defined in \eqref{G}), and so that for all $x\in\SX\setminus\{0\}$ it holds
\[
\mbox{either} \quad \|x_{m_L}\|\leq (1+\dt)\,\la\,\|x-\Phi\|,\quad
\mbox{or}\quad  |T\cap \Ga_{m_L}|> 2^{L-2},
\]
provided that $N+m_L<D$. Moreover, we can set
\[
\beta_0=2\,\ln\Big(\frac{8k_N(1+(1+\dt)\la)}\dt\Big).
\]
\end{theorem}
\begin{proof}
Let $n\geq0$ be such that $2^{n-1}<|T|\leq 2^n$. Now, for each $j=1,2,\ldots, n+1$, choose $A_j\subset T$ such that
\[
\|\Phi-\Phi_{A_j}\|=\min_{{A\subset T}\atop{|A|\leq 2^{j-1}}}\|\Phi-\Phi_A\|.
\]
Then, define $B_j=T\setminus A_j$. Picking the sets $A_j$ with smallest cardinality we may assume that 
$|A_1|\leq |A_2|\leq \ldots\leq |A_{n+1}|$ (although these sets may not be nested). 
Observe that, as special cases we can let
\[
(A_{n+1},B_{n+1})=(T,\emptyset)\mand (A_0,B_0)=(\emptyset, T).
\]
This construction implies the following
\Be\label{keyG}
\mbox{{\bf Key Fact:}\quad \emph{If $T'\subset T$ and $\|\Phi-\Phi_{T'}\|<\|\Phi_{B_j}\|$ then $|T'|>2^{j-1}$.}}
\Ee
This will be a crucial argument later to conclude the proof of the theorem.

Let $\beta>0$ be a large number to be determined later, and define
\[
\eta=e^{-\beta/2}\mand b=\frac1{2\eta}.
\]
 For the moment assume that $\beta$ is large enough so that $\eta<1/2$.
With that choice of $\beta$ we have $b>1$. Now, pick the first positive integer $L=L(b, \Phi)\in\SN$ such that
\Be
\|\Phi_{B_{j-1}}\|<b\,\|\Phi_{B_j}\|, \;j=1,2,\ldots, L-1,\mand \|\Phi_{B_{L-1}}\|\geq b\, \|\Phi_{B_L}\|.
\label{bG}
\Ee
Note that we could have $L=1$ if the first condition never holds, i.e. whenever $\|\Phi\|=\|\Phi_{B_0}\|\geq b\,\|\Phi_{B_1}\|$.
At the other extreme, we always have
\[
\|\Phi_{B_n}\|\geq \,b\, \|\Phi_{B_{n+1}}\|=0,
\]
which implies that $1\leq L\leq n+1$. Thus,
\[
2^{L-2}\leq 2^{n-1}<|T|,
\]
which is the first assertion in \eqref{LmLG}. Observe also that $A_L\not=\emptyset$, since otherwise 
we would have $A_j=\emptyset$, for all $j\leq L$, and hence $\Phi_{B_L}=\Phi_{B_{L-1}}=\Phi$, which would contradict the right hand side
of \eqref{bG}.

\

We now apply iteratively Theorem \ref{th_iterG}. Consider the numbers $m_0=0$ and 
\[
m_j=m_{j-1}+ \lfloor\beta \,G(|A_j|)\rfloor, \quad j=1,\ldots, L.
\]
Actually, to avoid trivial cases, we should restrict to $j=j_0,\ldots, L$, where $j_0$ is the first integer such that $|A_{j_0}|\not=0$
(and let $m_j=0$ for $j<j_0$). We also assume that $\beta$ is large enough so that 
\Be
\label{G1}
\beta \,G(1)\geq 1.
\Ee 
Observe that
\[
m_L=\sum_{j=j_0}^L \lfloor\beta \,G(|A_j|)\rfloor\leq \beta \,\sum_{j=1} G(2^{j-1}),
\]
since $G$ is increasing. This is the second inequality in \eqref{LmLG}.

Now, for each $j=j_0,\ldots, L$ we apply Theorem \ref{th_iterG} with $k=0$, $m=m_{j-1}$, $M=\lfloor\beta \,G(|A_j|)\rfloor$ and $A=A_j$
to obtain
\Beas
\|x_{m_j}\| & \leq & e^{-\frac{\lfloor\beta \,G(|A_j|)\rfloor}{G(|A_j|)}}\,\|x_{m_{j-1}}\|+\la\,\Big(\|\Phi-x\|+\|\Phi_{B_j}\|\Big)\\
& \leq & \eta\,\|x_{m_{j-1}}\|+\la\,\Big(\|\Phi-x\|+\|\Phi_{B_j}\|\Big),
\Eeas
using in the last line that $\lfloor a\rfloor\geq a/2$ if $a\geq1$.
Observe that the above inequalities hold trivially for $1\leq j<j_0$ (if there any such $j$) since
\[
\|x_{m_j}\|=\|x_0\|=\|x\|\leq \|x-\Phi\|+\|\Phi\|=\|x-\Phi\|+\|\Phi_{B_j}\|,
\]
and in this case $B_j=T$. Therefore, we can iterate the inequalities to obtain
\[
\|x_{m_L}\|\leq \eta^L\|x_{m_0}\|+\la\sum_{j=1}^L\eta^{L-j}\,\Big(\|\Phi-x\|+\|\Phi_{B_j}\|\Big).
\]
For the first summand we can also use
\Be
\|x_{m_0}\|=\|x\|\leq\|x-\Phi\|+\|\Phi\|=\|x-\Phi\|+\|\Phi_{B_0}\|.
\Ee
We now use the crucial assumption \eqref{bG}, that is,
\[
\|\Phi_{B_j}\|\leq b^{L-1-j}\,\|\Phi_{B_{L-1}}\|, \quad j=0,\ldots, L-1,\mand \|\Phi_{B_L}\|\leq b^{-1}\,\|\Phi_{B_{L-1}}\|,
\]
which inserted in the above expression gives
\Bea
\|x_{m_L}\| & \leq &  
\la\sum_{j=0}^L\eta^{L-j}\,\|\Phi-x\|\,+\,\la\,b^{-1}\,\sum_{j=0}^{L}(\eta b)^{L-j}\,\|\Phi_{B_{L-1}}\|\nonumber\\
& \leq & \frac\la{1-\eta}\,\|\Phi-x\|\,+\,\frac{\la\,b^{-1}}{1-\eta b}\,\|\Phi_{B_{L-1}}\| \nonumber\\
& = & \frac\la{1-\eta}\,\|\Phi-x\|\,+\,4\la\,\eta\,\|\Phi_{B_{L-1}}\|,\label{key1G}
\Eea
using in the last step the choice of $b=1/(2\eta)$.

On the other hand, we can use that $\Sigma_N$ satisfies property $\Au(k_N, D)$ to obtain the following estimate
\Bea
\|\Phi-\Phi_{T\cap \Ga_{m_L}}\| & \leq & k_N\,\|\Phi-\G_{m_L}(x)\|\,\leq\,k_N\,\big(\|\Phi-x\|+\|x_{m_L}\|\big)\nonumber\\
& \leq & k_N\,\Big[(1+\tfrac{\la}{1-\eta})\|\Phi-x\|\,+\,4\la\,\eta\,\|\Phi_{B_{L-1}}\|\Big].\label{key2G}
\Eea
At this point we wish to use the Key Fact in \eqref{keyG}. So we distinguish two cases.

\sline {\bf Case 1:} $\|\Phi-x\|< A\, \|\Phi_{B_{L-1}}\|$, for some $A>0$ to be determined. Then
\[
\|\Phi-\Phi_{T\cap \Ga_{m_L}}\| <\, k_N\,\Big[(1+\tfrac{\la}{1-\eta})\,A\,+\,4\la\,\eta\,\Big]\,\|\Phi_{B_{L-1}}\|.
\]
In this case, we wish to select $A$ and $\eta$ (and hence $\beta$) so that
\Be\label{key1aG}
k_N\,\Big[(1+\tfrac{\la}{1-\eta})\,A\,+\,4\la\,\eta\,\Big]\leq 1,
\Ee
which by the Key Fact would imply that
\[
|T\cap \Ga_{m_L}|>2^{L-2}.
\]
\sline {\bf Case 2:} $\|\Phi_{B_{L-1}}\|\leq A^{-1}\,\|\Phi-x\|$. In this case, using \eqref{key1G} we have
\[
\|x_{m_L}\|\leq \la\,\Big(\tfrac1{1-\eta}\,+\,4\eta\,A^{-1}\Big)\,\|\Phi-x\|.
\]
So, we wish to select $A$ and $\eta$ (hence $\beta$) such that
\Be
\label{key2aG}
\tfrac1{1-\eta}\,+\,4\eta\,A^{-1} \leq \,1+\dt.
\Ee
Overall, we have reduced the theorem to find numbers $A$ and $\eta$ so that \eqref{key1aG} and \eqref{key2aG} hold.
Writing $A=\eta \,B$, this amounts to find $B$ and $\eta$ so that
\[
\frac1{1-\eta}+\frac4{B}\leq\,1+\dt\mand k_N\,\eta\,\big[(1+\tfrac \la{1-\eta})\,B\,+\,4\la\big]\leq 1.
\]
This is clearly possible if $B$ is chosen sufficiently large and $\eta$ sufficiently small.
In order to make an explicit choice, we let $B=8/\dt$, so we need to select $\eta$ so that
\[
\frac1{1-\eta}\leq\,1+\dt/2\mand k_N\,\eta\,\big[(1+\tfrac \la{1-\eta})\,8\dt^{-1}\,+\,4\la\big]\leq 1
\] 
If we impose the first condition, the second one will hold provided
\[
k_N\,\eta\,\big[(1+(1+\dt/2)\la)\,8\dt^{-1}\,+\,4\la\big]=
8\,k_N\,\eta\,\big[\la+\tfrac{1+\la}\dt\big]\leq 1.
\]
That is, we can choose
\[
\eta\leq \min\Big\{(8k_N\,[\la+\tfrac{1+\la}\dt])^{-1}, \,\frac\dt{2+\dt}\Big\}=(8k_N\,[\la+\tfrac{1+\la}\dt])^{-1},
\]
with the last equality following easily from $k_N\geq1$ and $\la\geq1$. So, simplifying a bit we can choose
\[
\eta=\frac\dt{8k_N(1+(1+\dt)\la)},
\]
and using that $\eta=e^{-\beta/2}$, we find the expression
\[
\beta=2\,\ln (1/\eta) = 2\,\ln\Big[\frac{8k_N(1+(1+\dt)\la)}\dt\Big].
\]
We finally observe that \eqref{G1} is also satisfied, as in fact we have $G(1)\geq1$.
This is a simple consequence of
\[
\frac1{G(1)}=Q\Big(\frac{(1-\tfrac1\la)\tau}{2H(1)}\Big)\leq Q\Big(\frac1{H(1)}\Big)\leq Q(\|\phi\|)=Q(|F_\phi(\phi)|)\leq 1,
\]
with the second inequality due to \eqref{sn1} (for any $\phi\in\cD$), and the last one due to $\D(Q)$.
\end{proof}
\BR
In order to ensure that $\|x_{m_L}\|\leq 2\|x-\Phi\|$ we must choose $\la$ and $\dt$ so that $(1+\dt)\la=2$.
For instance, $\la=\sqrt2$ and $\dt=\sqrt{2}-1$ will give the value
\[
\beta= 2\,\ln\big[\frac{24k_N}{\sqrt2-1}\big]=2\,\ln\big(24(1+\sqrt2)k_N\big).
\]
\ER

\subsection{Step 3}\label{S4.3}
The next step is a slight generalization of Theorem \ref{th2G}, 
which corresponds to the special case $k=0$.

\begin{theorem}\label{th3G}
Let $(\SX,\|\cdot\|,\cD)$ and $1\leq N<D$ be such that 
\Benu
\item[(i)] $(\SX,\|\cdot\|,\cD)$ satisfies $\D(Q)$

\item[(ii)] $\Sigma_N$ satisfies property $\Au(k_N,D)$.

\item[(iii)]  $\Sigma_N$ satisfies property $\At(H,D)$.
\Eenu
 Given $\la>1$ and $\dt>0$, let
$\beta_0=\beta_0(\la, \dt, k_N)>0$ be as in Theorem \ref{th2G}, and let $\beta\geq \beta_0$. 
If $x\in\SX$ and $\Phi=\sum_{j\in T}a_j\phi_j\in \Sigma_N$ are not null, and if $k\in\SN_0$ is such that
\[
T_k:=T\setminus\Ga_k\not=\emptyset,\quad 
\mbox{(where $\Ga_k=\supp \G_k(x)$),}
\]
 then
there exist integers $L\in\SN$ and $m_L\geq k+1$ such that
\Be
2^{L-2} < |T_k|\mand m_L-k\leq \beta\,\sum_{j=1}^L G(2^{j-1}),
\label{LmLkG}
\Ee
and so that 
\[
\mbox{either} \quad \|x_{m_L}\|\leq (1+\dt)\,\la\,\|x-\Phi\|,\quad
\mbox{or}\quad  |T_k\cap \Ga_{m_L}|> 2^{L-2},
\]
provided that $N+m_L<D$. 
\end{theorem}
\begin{proof}
Apply the construction in the first part of Theorem \ref{th2G} to the vector $\Phi_{T_k}$ (instead of $\Phi$). 
So for $\eta $ and $b$ fixed as above, this gives an integer $L\in\SN$ such that $2^{L-2}<|T_k|$ and sets 
$A_j\subset T_k$ and $B_j=T_k\setminus A_j$ such that
the inequalities in \eqref{bG} hold.

At this point we let $m_0=k$ and consider
\[
m_j=m_{j-1}+ \lfloor\beta G(|A_j|)\rfloor, \quad j=j_0,\ldots, L,
\]
where $j_0$ is the first integer such that $|A_{j_0}|\not=0$. Otherwise we let $m_j=m_0=k$
when $1\leq j<j_0$. As before, this choice (and the size of the sets $A_j$) gives the second assertion in \eqref{LmLkG}.

Now, if $j_0\leq j\leq L$ we apply Theorem \ref{th_iterG} with $m=m_{j-1}$, 
 $M=\lfloor\beta G(|A_j|)\rfloor$ and $A=A_j$
to obtain
\Beas
\|x_{m_j}\| & \leq & e^{-\frac{\lfloor\beta G(|A_j|)\rfloor}{G(|A_j|)}}\,\|x_{m_{j-1}}\|+\la\,\Big(\|\Phi-x\|+\|\Phi_{B_j}\|\Big)
\nonumber\\
& \leq & \eta\,\|x_{m_{j-1}}\|+\la\,\Big(\|\Phi-x\|+\|\Phi_{B_j}\|\Big).
\Eeas
When $0\leq j<j_0$, we have instead
\Beas
\|x_{m_j}\|=\|x_k\|=\dist(x,[\phi_i]_{i\in\Ga_k}) & \leq &  \|x-\Phi_{T\cap\Ga_k}\|=\|x-\Phi+\Phi_{T_k}\| \\
& \leq & \|x-\Phi\|+\|\Phi_{B_j}\|,
\Eeas
since $A_j=\emptyset$ and hence $B_j=T_k$. Thus, we can proceed exactly as we did in \eqref{key1G} to obtain the same conclusion,
namely
\Be
\|x_{m_L}\|\leq \frac\la{1-\eta}\,\|\Phi-x\|\,+\,4\la\,\eta\,\|\Phi_{B_{L-1}}\|.
\Ee
On the other hand, using property $\Au(k_N, D)$ we obtain
\Bea
\|\Phi_{T_k}-\Phi_{T_k\cap \Ga_{m_L}}\| & \leq & k_N\,\|\Phi_{T_k}-\G_{m_L}(x)+\Phi_{T\cap \Ga_k}\|\,
 = \, k_N\,\|\Phi-x+x_{m_L}\|\nonumber\\
& \leq & k_N\,\big(\|\Phi-x\|+\|x_{m_L}\|\big)\nonumber\\
& \leq & k_N\,\Big[(1+\tfrac{\la}{1-\eta})\|\Phi-x\|\,+\,4\la\,\eta\,\|\Phi_{B_{L-1}}\|\Big].
\Eea
which is the analogous inequality to \eqref{key2G} in the previous theorem.

At this point one considers the same two cases as in the lines following \eqref{key2G}. Namely

\sline {\bf Case 1:} $\|\Phi-x\|< A\, \|\Phi_{B_{L-1}}\|$, with the same $A>0$ as in Theorem \ref{th2G}. This implies
\[
\|\Phi_{T_k}-\Phi_{T_k\cap \Ga_{m_L}}\| <\, k_N\,\Big[(1+\tfrac{\la}{1-\eta})\,A\,+\,4\la\,\eta\,\Big]\,\|\Phi_{B_{L-1}}\|
\,\leq\,\|\Phi_{B_{L-1}}\|,
\]
so by the construction of the sets $(A_j,B_j)$ and the Key Fact one obtains
\[
|T_k\cap \Ga_{m_L}|>2^{L-2}.
\]
\sline {\bf Case 2:} $\|\Phi_{B_{L-1}}\|\leq A^{-1}\,\|\Phi-x\|$. In this case, the same reasoning
as in Theorem \ref{th2G} gives
\[
\|x_{m_L}\|\leq \la\,\Big(\tfrac1{1-\eta}\,+\,4\eta\,A^{-1}\Big)\,\|\Phi-x\|\,\leq\,(1+\dt)\,\la\,\|\Phi-x\|.
\]
This completes the proof of Theorem \ref{th3G}.
\end{proof}

\subsection{Step 4: Conclusion of the proof of Theorem \ref{th_newG}}\label{S4.4}
This part of the proof requires substantial modifications compared to \cite{Tem14,Tem18},
so we present it in detail. 

Write $\la_1=(1+\dt)\la$, say with \[
\la=\sqrt{\la_1}>1\mand \dt=\sqrt{\la_1}-1>0.
\] 
The iterative process discussed in the previous subsections produces 
a positive constant $\beta= 2\,\ln\Big[\frac{8k_N(1+\la_1)}{\sqrt{\la_1}-1}\Big]$, and the following sequences of numbers

\Bi
\item there exist positive integers $L_1$ and $m_{L_1}$ such that
\Be
m_{L_1}\leq \beta \sum_{j=1}^{L_1} G(2^{j-1})\mand 2^{L_1-2}\leq |T|,
\label{mL1G}
\Ee
with the property that
\[
\mbox{either}\quad \|x_{m_{L_1}}\|\leq \la_1\,\|x-\Phi\|\quad \mbox{or}\quad |T\cap \Ga_{m_{L_1}}|\geq 2^{L_1-2}.
\]
In the first case one stops; if not one iterates and applies 
 Theorem \ref{th3G} with $k=m_{L_1}$, which implies
\item there exist positive integers $L_2$ and $m_{L_2}>m_{L_1}$ such that
\Be
m_{L_2}-m_{L_1}\leq \beta \sum_{j=1}^{L_2} G(2^{j-1})\mand 2^{L_2-2}\leq |T_{m_{L_1}}|,
\label{x2G}
\Ee
with the property that
\[
\mbox{either}\quad \|x_{m_{L_2}}\|\leq \la_1\,\|x-\Phi\|\quad \mbox{or}\quad |T_{m_{L_1}}\cap \Ga_{m_{L_2}}|\geq 2^{L_2-2}.
\]
Again, in the first case one stops; if not one applies iteratively Theorem \ref{th3G}, with values of $k=m_{L_i}$,
$i=2,\ldots, s-1$, until some step $s$, where can one ensure that
\item there are positive integers $L_s$ and $m_{L_s}>m_{L_{s-1}}$ such that
\Be
m_{L_s}-m_{L_{s-1}}\leq \beta \sum_{j=1}^{L_s} G(2^{j-1})\mand 2^{L_s-2}\leq |T_{m_{L_{s-1}}}|,
\label{xs1G}
\Ee
 where
\Be
\begin{cases}
\mbox{either}\quad \|x_{m_{L_s}}\|\leq \la_1\,\|x-\Phi\|,\quad\\
 \mbox{or}
\quad |T_{m_{L_{s-1}}}\cap \Ga_{m_{L_s}}|\geq 2^{L_s-2}
\quad \mbox{{\bf and}}\quad m_{L_s}\geq 2\beta \,\tG(2N).
\end{cases}
\label{xsG}
\Ee
\Ei
Here $G(n)$ denotes the sequence defined in \eqref{G}, and the notation $\tG(n)$ stands for the associated summing sequence as in \eqref{tw}.

In the first case of \eqref{xsG} one stops; if not, we shall show that  the greedy algorithm actually covers the whole set $T$, that is
\Be
|T\cap \Ga_{m_{L_s}}|\geq N.
\label{TsG}
\Ee
This would imply that $x_{m_{L_s}}=0$, and so we would also stop.

Let us prove \eqref{TsG}. Here we shall use the assumption that the sequence $G(n)$ in \eqref{G} is increasing and 1-quasi-convex.
Observe that 
\Bea
|T\cap \Ga_{m_{L_s}}| & = & |T\cap \Ga_{m_{L_1}}|+|T_{m_{L_1}}\cap \Ga_{m_{L_2}}|+\ldots+|T_{m_{L_{s-1}}}\cap \Ga_{m_{L_s}}|\nonumber\\
&  \geq &  2^{L_1-2}+2^{L_2-2}+\ldots+2^{L_s-2}.\label{aux1G}
\Eea
Now, by Lemma \ref{L1} and the inductive assumptions, see \eqref{x2G}, for each $i=1,\ldots,s$, we have
\Be
2\beta\,\tG(2^{L_i-1}) > \beta\sum_{j=0}^{L_i-1}G(2^j)=\beta\sum_{j=1}^{L_i}G(2^{j-1})\geq m_{L_i}-m_{L_{i-1}},
\label{aux2G}
\Ee
with the notation $m_{L_0}=0$. Thus, applying the (non-decreasing) function $2\beta \tG(2\cdot)$ to both sides of \eqref{aux1G}
and using part b) of Lemma \ref{L2} we obtain
\Beas
2\beta\,\tG\big(2|T\cap \Ga_{m_{L_s}}|\big)& \geq & 2\beta\,\tG(2^{L_1-1}+2^{L_2-1}+\ldots+2^{L_s-1})\\
& \geq & 2\beta\,\Big(\tG(2^{L_1-1})+\tG(2^{L_2-1})+\ldots+\tG(2^{L_s-1})\Big)\\
& > & m_{L_s}\,\geq\, 2\beta\tG(2N),
\Eeas
using in the last line \eqref{aux2G} and the second assertion in \eqref{xsG}. Since $\tG$ is increasing this implies
\[
|T\cap \Ga_{m_{L_s}}|>N,
\]
which proves \eqref{TsG}. 

Thus, the process will indeed end after $m_{L_s}$ iterations. We now estimate this number using the remaining conditions
in \eqref{xs1G}.
Since the last inequality in \eqref{xsG} occurs for the first time at step $s$, we must have
\[
m_{L_{s-1}}<2\beta\tG(2N).
\]
Thus, 
\Beas
m_{L_s} & = & m_{L_{s-1}}\,+\,\big(m_{L_s}-m_{L_{s-1}}\big)\\
& \leq &  2\beta \tG(2N)\,+\,\big(m_{L_s}-m_{L_{s-1}}\big)\\
& \leq &  2\beta \tG(2N)\,+\,\beta\,\sum_{j=1}^{L_s} G(2^{j-1})\\
& \leq &  2\beta \tG(2N)\,+\,2\beta\,\tG(2^{L_s-1})\,\leq\,4\beta\,\tG(2N).
\Eeas
Therefore, using also part a) of Lemma \ref{L2}, we see that \eqref{xNG} will be true with 
\[
\oldphi(N)=m_{L_s}\leq 4\beta\,G(2N)\,=\,8\,\ln\Big[\frac{8k_N(1+\la_1)}{\sqrt{\la_1}-1}\Big]\,G(2N),
\]
as asserted in \eqref{phiHG}.
\ProofEnd

\section{An aplication: WCGA in $L^p(\log\,L)^\al$ spaces}
\setcounter{equation}{0}\label{S_Lpal}

\subsection{Property $\D(Q)$ in $L^p(\log\,L)^\al$}

In this section we shall apply Theorem \ref{th_newG} in the case when 
\[
\SX=L^p(\log\,L)^\al,\quad \mbox{$1<p<\infty$, $\al\in\SR$.}
\]
Following \cite[Definition IV.6.11]{BS}, this is the set of all measurable $f:\SR^d\to\SC$ such that
\[
\Big\|f(x)\,\big|\log(2+|f(x)|)|^\al\Big\|_{L^p(\SR^d)}<\infty.
\]
These classes satisfy the elementary inclusions
\[
L^p(\log L)^{|\al|}\subset L^p\subset L^p(\log L)^{-|\al|}.
\]
We shall regard $\SX$ as an Orlicz space $L^\Phi$ associated with the function
\Be
\Phi(t)=\int_0^t s^{p-1}\,\big(\log(c+s)\big)^{\al p}, \quad t\geq0,
\label{Phipa}
\Ee
which for 
a sufficiently large $c>1$ is a (smooth) Young function\footnote{A Young function is a convex, non-decreasing function 
$\Phi:[0,\infty)\to[0,\infty]$ such that $\Phi(0^+)=0$ and $\lim_{t\to+\infty}\Phi(t)=+\infty$. See \cite[\S I.3]{RR} or 
\cite[\S IV.8]{BS} for properties and various equivalent definitions.}.
The corresponding (Luxemburg) norm is then defined by
\[
\|f\|_{L^\Phi}=\inf\Big\{\la\mid \int_{\SR^d}\Phi(|f(x)|/\la)\,dx\leq 1\Big\}.
\]
Let $\Psi$ be the complementary function\footnote{The complementary function of a Young function $\Phi(t)$
is defined by $\Psi(s)=\sup_{t\geq0}\big(ts-\Phi(t)\big)$. See \cite[\S I.3]{RR} or 
\cite[\S IV.8]{BS} for properties and other equivalent definitions.} of $\Phi(t)$. Then
it is known that $(L^\Phi)^*=L^\Psi$ (isometrically,
when the latter space is endowed with the Orlicz norm); see \cite[Corollary IV.8.15]{BS}.
In these examples it is not difficult to check that 
\Be
\Phi(t)\,\approx\, t^p\,\big(\log(e+t)\big)^{\al p}, \quad t\geq0,
\label{Phita}
\Ee
and 
\Be
\Psi(t)\,\approx \,t^{p'}\,\big(\log(e+t)\big)^{-\al p'}, \quad t\geq0;
\label{LPsi}
\Ee
see e.g. \cite[Theorem I.7.2]{KR61}.

We recall how the norming functional $F_f$ of a (normalized) element 
$f\in L^\Phi$ is defined; see \cite[Theorem 18.5]{KR61}.
Let $P(z)$, $z\in\SC$, be an extension of $P(t)=\Phi'(t)$, $t\geq0$, such that $\sign(zP(z))=1$, $\forall z\not=0$.
Then  $F_f$ is explicitly given by
\[
F_f(g)=\frac{\Ds\int_{\SR^d} P(f(x))\,g(x)\,dx}{\Ds\int_{\SR^d} P(f(x))\,f(x)\,dx}.
\] 
In our case of interest we will have $P(z)=|z|^{p-2}\,\bar{z}\,\big(\log(c+|z|)\big)^{\al p}$, $z\in\SC$, and hence
\Be
F_f(g)=\tfrac1{A(f)}\int_{\SR^d}|f(x)|^{p-2}\,\overline{f(x)}\,g(x)\,\big(\log(c+|f(x)|)\big)^{\al p}\,dx,
\label{Ff}
\Ee
with $A(f)=\int_{\SR^d}|f|^{p}\,(\log(c+|f|))^{\al p}\,dx$.

\

The moduli of smoothness and convexity for Orlicz spaces $L^\Phi$
have been studied in \cite{MT75, Fi76}. According to \cite[Theorem 1]{MT75}, there exists a Young function $\bar\Phi$, equivalent to $\Phi$,
such that 
\Be
\rho_{L^{\bar\Phi}}(t)\lesssim \sup_{u\in[s,1],\,v>0}\frac{t^2\,\Phi(uv)}{u^2\,\Phi(v)}
\mand \dt_{(L^{\bar\Phi})^*}(s)\gtrsim \inf_{u\in[s,1],\,v>0}\frac{s^2\,\Psi(uv)}{u^2\,\Psi(v)}
\label{rhodt}
\Ee
under suitable doubling conditions in $\Phi(t)$ and $\Psi(t)$ (which are always held in the cases 
considered in \eqref{Phita} and \eqref{LPsi}). 
Moreover, our specific examples satisfy the regularity conditions stated in \cite[Proposition 19]{Fi76},
so one may actually take $\bar\Phi=\Phi$.
 
Therefore, inserting into \eqref{rhodt} the expressions for $\Phi$ and $\Psi$ from \eqref{Phita} and \eqref{LPsi},
and performing some straightforward computations, one obtains the following result.
Here we use the standard notation $\al=\al_+-\al_-$, where 
\[
\al_+=\max\{\al,0\}\mand \al_-=\max\{-\al,0\}.
\]

\begin{proposition}\label{P_rhodt}
Let $1<p<\infty$ and $\al\in\SR$, and let $\SX=L^p(\log L)^\al$ as above.
Then, for the Luxemburg norm associated with $\Phi(t)$ in $\SX$ it holds
\Bi
\item if $p>2$ then
\[
\rho_{\SX}(t)\,\lesssim\, t^2\mand \dt_{\SX^*}(s)\,\gtrsim \,s^2,
\]
\item if $1<p\leq 2$ then 
\[
\rho_{\SX}(t)\,\lesssim \,
t^p\,(\log (e+\tfrac1t))^{p\,\al_-}\mand
\dt_{\SX^*}(s)\,\gtrsim \, \frac{s^{p'}}{(\log (e+\frac1s))^{p'\,\al_-}}.
\]
\Ei
%
\end{proposition}
In view of the discussion in \S\ref{S_DQ}, we then obtain the following.

\begin{corollary}\label{C_Q}
Let $1<p<\infty$ and $\al\in\SR$, and let $\SX=L^p(\log L)^\al$ as above.
Then, for every normalized dictionary $\cD$ in $\SX$, property $\D(Q)$ holds with 
\Be
Q(s):=\,\begin{cases} \frac{c_0\,s^{p'}}{(\log (e+\frac1s))^{p'\,\al_-}} & \mbox{if {\small $1<p\leq 2$}}\\
c_0\,s^2 & \mbox{if {\small $p>2$,}}
\end{cases}
\label{dtstar}
\Ee
for a suitably small constant $c_0>0$.
\end{corollary}

\subsection{The Haar system in $L^p(\log\,L)^\al$}

Next we consider the dictionary $\cD=\{\psi_j\}_{j=1}^\infty$ in $\SX=L^p(\log L)^\al$ given by the normalized Haar basis in $\SR^d$
(or any sufficiently smooth wavelet basis). These bases are unconditional, so in this case property $\Au(k_N,D)$
will hold with $k_N=O(1)$ (and any $D<\infty$). It remains to verify property $\At(H,D)$. As mentioned
in Lemma \ref{L_A3H}, this property holds with
\Be
H(N)=\sup_{|A|\leq N, |\e_j|=1}\,\big\|\sum_{j\in A} \e_j\psi^*_j\big\|_{\SX^*}
\label{HN1}
\Ee
(also for all $D<\infty$). Here $\{\psi^*_j\}$ is the dual dictionary, which is again the Haar basis, this
time normalized in $\SX^*$. Since the basis is unconditional, the parameter $H(N)$ is equivalent to the
\emph{upper democracy function} of the dual space $\SX^*$, that is
\[
H(N)\approx h_{\SX^*}(N):=\sup_{|A|\leq N} \big\|\sum_{j\in A} \psi^*_j\big\|_{\SX^*}.
\]
Democracy functions, for the Orlicz classes $L^\Phi$, were studied in \cite{GHM08},
where it was proved that, if the Boyd indices of $\Phi$ are non trivial, then
\[
h_{L^\Phi}(N)\approx \sup_{s>0}\frac{\phi(Ns)}{\phi(s)},
\]
where $\phi(t):=1/\Phi^{-1}(1/t)$ is the fundamental function of $L^\Phi$.

In our case of interest, where $\Phi(t)$ satisfies \eqref{Phita}, we have
\[
\phi(t)\approx t^{1/p}\,\big(\log(e+\tfrac1t)\big)^{\al}\mand
h_{L^\Phi}(N)\approx N^{1/p}\,\big(\log(e+N)\big)^{\al_-};
\]
see \cite[Proposition 3.4]{GHM08}. Thus, for the dual space $\SX^*=L^\Psi$ with $\Psi$ as in \eqref{LPsi}
we have
\[
h_{L^\Psi}(N)\approx N^{1/p'}\,\big(\log(e+N)\big)^{\al_+}.
\]
Overall we conclude that property $\At(H)$ holds with
\Be
H(N)\,\approx \, N^{1/p'}\,\big(\log(e+N)\big)^{\al_+}.
\label{Hpsi}
\Ee
Thus, combining \eqref{dtstar} and \eqref{Hpsi}, we see that, for $p>2$ we have
\[
G(N)=\frac1{Q(\frac {c(\tau)}{H(N)})} \,\approx\,H(N)^2\,\approx\,N^\frac2{p'}\,\big(\log(e+N)\big)^{2\al_+}
\]
while for $1<p\leq 2$ we have
\[
G(N)=\frac1{Q(\frac {c(\tau)}{H(N)})} \,\approx\,H(N)^{p'}\,\Big(\log(e+H(N))\Big)^{p'\al_-}\,
\approx\,N\,\big(\log(e+N)\big)^{p'|\al|}.
\]

\subsection{Proof of Theorem \ref{th_Lpal}.a}
Collecting the values of the parameters $G(N)$ obtained
in the previous subsection, and inserting them into Theorem \ref{th_newG}, we deduce the 
first assertions \eqref{foldphiN} and \eqref{phiLpal} in Theorem \ref{th_Lpal}.


\subsection{Proof of Theorem \ref{th_Lpal}.b} 
\setcounter{footnote}{0}

We shall use the following result whose proof can be found in \cite[Lemma 3.1]{GHM08}. 
For simplicity in the notation, we assume in this section
that the underlying space $\SR^d$ has dimension $d=1$.

\begin{lemma}\label{L_ghm}
Let $\SX=L^\Phi(\SR)$ be an Orlicz space with non-trivial Boyd indices, and let $\cD=\{h_I\}$ be the (normalized) Haar basis in $\SX$.
Then, if $A$ is a finite collection of disjoint dyadic intervals with the same size $s$, then
\[
\big\|\sum_{I\in A} h_I\big\|_{L^\Phi} \,\approx\, \frac{\phi(|A|s)}{\phi(s)},
\]
where $\phi(t)$ is the fundamental function of $\SX$.
\end{lemma}

We now show the lower bound for the function $\psi(N)$ 
stated in \eqref{fpsiN}.



\begin{proof}[Proof of \eqref{fpsiN}.] 
Write $\cD=\{h_I\}$ where $h_I$ is the (normalized) Haar function supported in $I$, and $I$ runs over all dyadic 
intervals
 in $\SR$.
Pick any two collections $A$ and $B$, of pairwise disjoint dyadic intervals with cardinalities $|A|=N$ and $|B|=M$, such that
\[
|I|=1\quad \mbox{if $I\in A$,}\mand |I|=1/M\quad \mbox{if $I\in B$.}
\]
For instance, we could take 
\[
A=\{[n,n+1)\mid n=1,\ldots, N\}\mand B=\{J\subset[0,1)\mid |J|=2^{-m}\},
\]
with $M=2^m$. For $b>0$ to be determined, consider the function
\Be
f=f_1+f_2=\sum_{I\in A}h_I+b\sum_{I\in B}h_I.
\label{f12}
\Ee
Using Lemma \ref{L_ghm} and $\phi(t)\,\approx\,t^{1/p}\,\big(\log(e+\frac1t)\big)^\al$, observe that 
\[
\|f_1\|\approx \phi(N)\approx N^{1/p},\mand \|f_2\|\approx \frac b{\phi(1/M)}\approx \frac{b\, M^\frac1p}{(\log(e+M))^\al}.
\]
Also, since $\|h_I\|_{L^\Phi}=1$ we have
\Be
|h_I(x)|=\frac1{\phi(|I|)}\bone_I(x).
\label{faux0}
\Ee
In particular, if $I\in B$ we have
\Be
|f(x)|=b\,|h_I(x)|=\frac b{\phi(1/M)}\approx \|f_2\|\lesssim \|f\|,\quad x\in I,
\label{faux1}
\Ee
and similarly, if $I\in A$ we have
\Be
|f(x)|=1\lesssim \|f_1\|\lesssim \|f\|, \quad x\in I.
\label{faux2}
\Ee
Using the formula for the norming functional in \eqref{Ff} we see that
\[
|F_f(h_I)|=\,\frac1{C(f)}\,\times\,\begin{cases}\Ds
\int |h_I|^p\,\Big(\log(c+|f(x)|/\|f\|)\Big)^{\al p}\,dx, & \quad I\in A,\\
\Ds{b^{p-1}}\int |h_I|^p\,\Big(\log(c+|f(x)|/\|f\|)\Big)^{\al p}\,dx, & \quad I\in B,\\
0, & \quad I\not\in A\cup B.
\end{cases}
\]
In view of \eqref{faux1} and \eqref{faux2}, the logarithmic factors inside the integrals are approximately constant,
so can be disregarded.
Also, \eqref{faux0} implies
\[
\int |h_I|^p\,dx \,=\, \frac{|I|}{\phi(|I|)^p}\,\approx\,\frac1{\big(\log(e+|I|^{-1})\big)^{\al p}},
\]
so we have
\Be
|F_f(h_I)|\,\approx\,\tfrac1{C(f)}, \;\;\mbox{$I\in A$},\mand 
|F_f(h_I)|\,\approx\,\tfrac1{C(f)} \,\frac{b^{p-1}}{\big(\log(e+M)\big)^{\al p}},\;\;\mbox{$I\in B$}.
\label{faux3}
\Ee
Thus, the above quantities are approximately the same provided we choose
\Be
\label{bM}
b\,=\,c_1\,\big(\log(e+M)\big)^{\al p'}.
\Ee
Therefore, if $c_1>0$ is chosen properly, the WCGA, $\G_n(f)$, can be formed either by selecting consecutive elements $I$ from $A$
(if $n\leq N$), 
or by selecting consecutive elements $I$ from $B$ (if $n\leq M/2$). To verify these assertions one should note that the equivalences in \eqref{faux3}
remain also true\footnote{In the latter case, we have restricted to $n\leq M/2$ to ensure that
\eqref{faux1} continues to hold when $f$ is replaced by $f-\G_n(f)$. Indeed, in such case 
one would use that $\|\sum_{I\in B'}h_I\|\approx 1/\phi(1/M)$, when $B'\subset B$ with $|B'|\geq M/2$, by Lemma \ref{L_ghm}.}
 when $f$ is replaced by the remainder $f-\G_n(f)$.

\

So suppose now that \eqref{fpsiN} holds. If $\al\geq0$, we let $N=\psi(M)$, and in view of the previous comment we can select $c_1$
such that $\G_N(f)=f_1$. Then
\[
\|f_2\|=\|f-\G_{\psi(M)}(f)\|\leq 2\sigma_M(f)\leq \|f_1\|,
\]
which in view of \eqref{bM} and \eqref{f12} implies 
\[
c_1^p\,M\,(\log(e+M))^{\al p'}\,=\,\frac{b^p\, M}{(\log(e+M))^{\al p}}\,\approx \,\|f_2\|^p\lesssim \|f_1\|^p\approx N=\psi(M).
\]
This proves the assertion in the Theorem when $\al\geq0$.

If $\al\leq0$, then we take $M=2\psi(N)$, and select $c_1$
such that $\G_{M/2}(f)=b\sum_{I\in B'}h_I$, for some $B'\subset B$ with $|B'|=M/2$. Then
\[
\|f_1\|\lesssim\|f_1+b\sum_{I\in B\setminus B'}h_I\|=\|f-\G_{\psi(N)}(f)\|\leq 2\sigma_N(f)\leq \|f_2\|,
\]
which this time implies 
\[
N\approx \|f_1\|^p\lesssim \|f_2\|^p\approx \,M\,(\log(e+M))^{\al p'}\,.
\]
Solving for $M$ this gives
\[
\psi(N)=M/2\,\gtrsim\,\frac{N}{(\log(e+N))^{\al p'}}\,=\,N\,\big(\log(e+N)\big)^{|\al|\,p'}.
\]
This establishes \eqref{fpsiN}, and therefore completes the proof of Theorem \ref{th_Lpal}.
\end{proof}

\section{WCGA for trigonometric system in $L^p(\log\,L)^\al$}
\setcounter{equation}{0}\label{S_trig}

In this section we give a second application of Theorem \ref{th_newG},
this time to the trigonometric system in the torus $\ST\equiv[-\pi,\pi)$, that is,
\[
\cD=\cT:=\{e^{inx}\}_{n\in\SZ}.
\]
So, from now on, all functions $f\in L^p(\log\,L)^\al$ are understood as defined in $\ST$. Otherwise, we regard $L^p(\log\,L)^\al$ as 
an Orlicz space $L^\Phi$ in the same sense as in \S\ref{S_Lpal}. 
Since \cite{MT75} covers also this setting, the estimates for the moduli of convexity and smoothness
in Proposition \ref{P_rhodt} remain true, and so does the estimate \eqref{dtstar} for the function $Q(s)$ in Corollary
\ref{C_Q}. 

\

We still have to compute the parameters $k_N$ and $H(N)$. To do so, we shall make use of the following interpolation lemma.

\begin{lemma}\label{L_p22}
Consider the Young function $\bPhi(t)=t^p\,\big(\log(c+t)\big)^\al$, for some $c\geq e$.
Assume that 
\Be
2<p<\infty\;\;\mbox{and}\;\;\al\in\SR,\quad\mbox{or}\quad p=2\;\;\mbox{and}\;\;\al\geq0.
\label{p22}
\Ee 
Then, 
\Be
\label{interpol}
\|f\|_{L^\bPhi}\leq \|f\|_\infty^{1-\frac2p}\,\Big(\log\big(c+\big[\tfrac{\|f\|_\infty}{\|f\|_2}\big]^\frac p2\big)\Big)^\al\,\|f\|_2^{\frac 2p},
\quad\quad \forall\,f\in L^\infty(\ST).
\Ee
\end{lemma}
\begin{proof}
We may assume that $\|f\|_2=1$. Define the functions
\Be
\label{abt}
a(t)=t^\frac 2p\,\big(\log(c+t^\frac 2p)\big)^{-\al}\mand b(t)=\frac t{a(t)}=
t^{1-\frac 2p}\,\big(\log(c+t^\frac 2p)\big)^{\al}, \quad t>0.
\Ee
By the lattice property of the Luxemburg norm in $L^\bPhi$ we have
\[
\|f\|_{L^\bPhi}  = \big\|a(f)b(f)\big\|_{L^\bPhi}\leq \,\big\|b(f)\big\|_{L^\infty}\,\big\|a(f)\big\|_{L^\bPhi}\,
 = \,  b\big(\|f\|_\infty\big)\,\big\|a(f)\big\|_{L^\bPhi},
\]
using that $b(t)$ is increasing under the conditions in \eqref{p22}. So, it suffices to show that
\[
\int_\ST\bPhi\big(a(|f(x)|)\big)\,dx\,\leq\,1,
\] 
as this will imply that $\big\|a(f)\big\|_{L^\bPhi}\leq 1$. Write
\Beas
\int\bPhi\big(a(|f|)\big)\,dx & = & \int a(|f|)^p\,\Big[\log\big(c+a(|f|)\big)\Big]^{\al p}\,dx.
\Eeas
Observe that, regardless of the sign of $\al\in\SR$, we always have
\[
\Big[\log\big(c+a(|f|)\big)\Big]^{\al p}=\Bigg[\log\Big(c+\frac{|f|^\frac2p}{[\log(c+|f|^{2/p})]^{\al}})\Big)\Bigg]^{\al p}\leq 
\Big[\log\Big(c+|f|^\frac2p\Big)\Big]^{\al p}.
\]
Thus,
\[
\int\bPhi\big(a(|f|)\big)\,dx\,\leq\, \int a(|f|)^p\,\Big[\log\big(c+|f|^\frac2p\big)\Big]^{\al p}\,dx\,=\,\int|f|^2=1.
\]
\end{proof}

\BR
Observe that, when the indices $p$ and $\al$ satisfy \eqref{p22}, then it holds
\Be
L^p(\log L)^\al\hookrightarrow L^2(\ST).
\label{LplogL2}
\Ee
This is easily proved using that $t^2\lesssim \Phi(t)$ for $t\geq1$, since
\[
\int_\ST|f|^2=\int_{\{|f|<1\}}|f|^2+\int_{\{|f|\geq1\}}|f|^2\leq 1+c'\int\Phi(|f|)\,dx<\infty.
\]
Likewise, by duality, one proves that $L^{2}(\ST)\hookrightarrow L^p(\log L)^\al$ when
\Be
1<p<2\;\;\mbox{and}\;\;\al\in\SR,\quad\mbox{or}\quad p=2\;\;\mbox{and}\;\;\al\leq0.
\label{p11}
\Ee 
\ER

\subsection{Property $\At$ for $\cT$ in $L^p(\log L)^\al$}

\begin{lemma}\label{L_demT}
Let $1<p<\infty$ and $\al\in\SR$.
Then, for all $|\e_n|\leq 1$ and all $A\subset\SZ$ with $|A|\leq N$ it holds
\Be
\big\|\sum_{n\in A}\e_n e^{inx}\big\|_{L^p(\log L)^\al}\lesssim\max\Big\{N^{1/2}, \;N^{1-\frac1p}\,\big(\log(e+N)\big)^\al\Big\}.
\label{demA}
\Ee
\end{lemma}
\begin{proof}
When $p$ and $\al$ satisfy \eqref{p11}, the right hand side of  \eqref{demA} is $\approx N^{1/2}$, so the
assertion follows from the inclusion $L^{2}\hookrightarrow L^p(\log L)^\al$.
On the other hand, if $p$ and $\al$ satisfy \eqref{p22}, then by Lemma \ref{L_p22} 
we have
\Be
\|f\|_{L^p(\log L)^\al}\,\lesssim\,
b\Big(\tfrac{\|f\|_\infty}{\|f\|_2}\Big)\,\|f\|_2,
\label{bbt}
\Ee
where $b(t)=t^{1-\frac 2p}\,\big(\log(c+t^\frac 2p)\big)^{\al}$.
Applying this to $f=\sum_{n\in A}\e_n e^{inx}$, and using that $b(t)$ is increasing and
\[
{\|f\|_\infty}/{\|f\|_2}\,\leq\, N/\sqrt{N}=\sqrt N, 
\] 
 one easily obtains \eqref{demA}.
\end{proof}
\BR\label{R5.11}
The upper bounds in \eqref{demA} cannot be improved, even when all signs $\e_n=1$. Indeed, 
if one considers the Dirichlet kernel $D_N(x)=\sum_{|n|\leq N} e^{inx}$, then we have
\Be
\label{DN}
\|D_N\|_{L^p(\log L)^\al}\,\approx\, N^{1-\frac1p}\,\big(\log(e+N)\big)^\al;
\Ee
see e.g. \cite[Lemma 3.1]{PaWo22}. On the other hand, if $A$ is a lacunary set (say, $A=\{2^j\}_{j=1}^N$),
then 
\[
\big\|\sum_{n\in A} e^{inx}\big\|_{L^p(\log L)^\al}\,\approx\, \sqrt N.
\]
Indeed, this is easily obtained from a similar result for all the $L^q$ spaces, $0<q<\infty$, and the inclusions
$L^{p+\e}\hookrightarrow L^p(\log L)^\al\hookrightarrow L^{p-\e}$.
\ER

\begin{corollary}\label{C_Htrig}
Let  $1<p<\infty$ and $\al\in\SR$. Let $\SX=L^p(\log L)^\al$ and $\cD=\{e^{inx}\}_{n\in\SZ}$ in $\ST$.
Then, property $\At(H)$ holds with
\[
H(N)\,\approx\, \max\Big\{N^{1/2}, \;N^{\frac1p}\,\big(\log(e+N)\big)^{-\al}\Big\}.
\]
\end{corollary}
\begin{proof}
Apply Lemmas \ref{L_A3H} and \ref{L_demT}, and the duality relation $\SX^*=L^{p'}(\log L)^{-\al}$.
\end{proof}

\subsection{Property $\Au$ for $\cT$ in $L^p(\log L)^\al$}

Given a finite set $A\subset\SZ$, we denote 
\[
S_A(g)=\sum_{n\in A}\hg(n)e^{inx},
\]
where $\hg(n)$, $n\in\SZ$, are the Fourier coeffients of $g\in L^1(\ST)$. As noticed in \cite[Lemma 2.15]{DGHKT21}, 
property $\Au(k_N)$ holds trivially when we let
\Be
k_N=\sup_{|A|\leq N}\|S_A\|_{L^p(\log L)^\al\to L^p(\log L)^\al}.
\label{kNSA}
\Ee
In this section we compute this last expression.

\begin{lemma}\label{L_kN}
Let $2<p<\infty$ and $\al\in\SR$, or $p=2$ and $\al\geq0$.
Then, 
\Be
\sup_{|A|\leq N}\big\|S_A\big\|_{L^p(\log L)^\al\to L^p(\log L)^\al}\,\lesssim \;N^{\frac12-\frac1p}\,\big(\log(e+N)\big)^\al.
\label{kNA}
\Ee
\end{lemma}
\begin{proof}
Let $g\in L^p(\log L)^\al$.
Using the inequality in \eqref{bbt} 
from the previous section, applied to $f=S_A(g)$ we see that
\[
\|S_A(g)\|_{L^p(\log L)^\al}\,\lesssim\,
b\Big(\tfrac{\|S_A(g)\|_\infty}{\|S_A(g)\|_2}\Big)\,\|S_A(g)\|_2.
\]
Now, 
\[
\|S_A(g)\|_\infty\leq |A|^{1/2}\,\big(\sum_{n\in A}|\hg(n)|^2\big)^{1/2}\leq \sqrt N\,\|S_A(g)\|_2,
\]
so using that $b(t)$ is increasing we obtain
\[
\|S_A(g)\|_{L^p(\log L)^\al}\,\lesssim\,
b\big(\sqrt N\big)\,\|S_A(g)\|_2.
\]
On the other hand, the inclusion in \eqref{LplogL2} gives
\[
\|S_A(g)\|_2\leq \|g\|_2\lesssim \|g\|_{L^p(\log L)^\al}.
\]
Thus, we obtain
\[
\|S_A\|\lesssim b(\sqrt N) \approx \,N^{\frac12-\frac1p}\,\big(\log(e+N)\big)^\al.
\]
\end{proof}

Since $S_A^*=S_A$, by duality one obtains the following complementary result.
\begin{lemma}\label{L_kN2}
Let $1<p<2$ and $\al\in\SR$, or $p=2$ and $\al\leq0$. 
Then, 
\Be
\sup_{|A|\leq N}\|S_A\|_{L^p(\log L)^\al\to L^p(\log L)^\al}\lesssim\;N^{\frac1p-\frac12}\,\big(\log(e+N)\big)^{-\al}.
\label{kNA2}
\Ee
\end{lemma}

\BR
The estimate in \eqref{kNA2} is best possible (and by duality, also \eqref{kNA}). One can prove this
by noticing that there exist choices of signs $\pm1$ such that
\Be
\label{pm}
\big\|\sum_{|n|\leq N}\pm e^{inx}\big\|_{L^p(\log L)^\al}\gtrsim \sqrt N.
\Ee
This last assertion can be easily obtained from a similar property of the $L^q$-spaces, and the inclusions
at the end of Remark \ref{R5.11}. From \eqref{pm}, there will be a set $A\subset[-N,N]$, either corresponding
to the positive or the negative signs,  so that
\[
\big\|\sum_{n\in A}e^{inx}\big\|_{L^p(\log L)^\al}\gtrsim \tfrac12\,\sqrt N.
\]
Thus, omitting the subindices $L^p(\log L)^\al$ from the norms, we have
\[
\|S_A\|\geq \|S_A(D_N)\|/\|D_N\|= \big\|\sum_{n\in A}e^{inx}\big\|/\|D_N\|\gtrsim N^{\frac1p-\frac12}\,\big(\log(e+N)\big)^{-\al},
\]
using \eqref{DN} in the last step.
\ER

\begin{corollary}\label{C_kNtrig}
Let  $1<p<\infty$ and $\al\in\SR$. Let $\SX=L^p(\log L)^\al$ and $\cD=\{e^{inx}\}_{n\in\SZ}$ in $\ST$.
Then, property $\Au(k_N)$ holds with
\[
k_N\,\approx\, \begin{cases}
N^{\frac12-\frac1p}\,\big(\log(e+N)\big)^\al & \mbox{if $p>2$, or $p=2$ and $\al\geq0$}\\
N^{\frac1p-\frac12}\,\big(\log(e+N)\big)^{-\al} & \mbox{if $1<p<2$, or $p=2$ and $\al\leq0$}.\\
\end{cases}
\]
\end{corollary}
\begin{proof}
Apply Lemmas \ref{L_kN} and \ref{L_kN2} to the expression in \eqref{kNSA}.
\end{proof}

\subsection{WCGA for $\cT$ in $L^p(\log L)^\al$}

Combining the estimates from the previous subsections, we obtain the following.

\begin{theorem}\label{th3}
Let  $1<p<\infty$ and $\al\in\SR$. Let $\SX=L^p(\log L)^\al$ and $\cD=\{e^{inx}\}_{n\in\SZ}$ in $\ST$.
Then, there exists a constant $C>1$ such that the WCGA satisfies
\[
\Big\|f-\G_{\oldphi(N)}(f)\Big\|_{L^p(\Log L)^\al}\leq \,2\,\sigma_N(f) 
,\quad \forall\,f\in L^p(\log L)^\al,\;N\geq2,
\]
where 
\Be
\label{phitrig}
\oldphi(N)\,=\,C\,\begin{cases}
N\,\log N & \mbox{when $p>2$}\\
N\,\log\log N & \mbox{when $p=2$ and $\al>0$}\\
N & \mbox{when $p=2$ and $\al=0$}\\
N\,(\log N)^{4\al_-}\,\log\log N & \mbox{when $p=2$ and $\al<0$}\\
N^{p'-1}\,\frac{(\log N)^{p'\al_-}}{(\log N)^{p'\al}}\,\log N &   \mbox{when $1<p<2$}.\\
\end{cases}
\Ee
\end{theorem}
\begin{proof}
Combine Theorem \ref{th_newG}, with the estimates for $Q(t)$, $H(N)$ and $k_N$ in 
Corollaries \ref{C_Q}, \ref{C_Htrig} and \ref{C_kNtrig}.
\end{proof}
\BR
The necessity of the log factors and the powers in
the above expression of $\oldphi(N)$ is not known, even in the case $\al=0$ (except, of course,
if $\SX=L^2$). See \cite[Open Question 8.2]{Tem18}.
\ER

\section*{ Acknowledgments }{ 
Research partially supported by grants {\small MTM2017-83262-C2-2-P,
PID2019-105599GB-I00} from \emph{Ministerio de Ciencia e Innovaci\'on} (Spain), and grant 20906/PI/18 from 
\emph{Fundaci\'on S\'eneca} (Regi\'on de Murcia, Spain).

The author wishes to thank E. Hern\'andez and D. Kutzarova for useful comments at different stages of this work.}

\bibliographystyle{plain}

\begin{thebibliography}{1}

\bibitem{BS}
   \textsc{C. Benett and R.C. Sharpley},
   \emph{Interpolation of operators}. Academic Press, 1988.

\bibitem{DGHKT21} \textsc{S. Dilworth, G. Garrig\'os, E. Hern\'andez, D. Kutzarova,
V. Temlyakov}, Lebesgue-type inequalities in greedy approximation. Jour. Funct. Anal. {\bf 280} (5) (2021), 108885


\bibitem{Fi76}
\textsc{T. Figiel}, On the moduli of convexity and smoothness. Studia Math {\bf 56} (1976), 121--155.

\bibitem{GHM08} \textsc{G. Garrig\'os, E. Hern\'andez, J.M. Martell},
Wavelets, Orlicz spaces, and greedy bases. 
Appl. Comput. Harmon. Anal. 24 (1) (2008), 70--93.


\bibitem{KR61}
\textsc{M. Krasnosel'skii, J. Rutickii}. \emph{Convex functions and Orlicz spaces}, 
Noordhoff Ltd., Groningen 1961.

\bibitem{LZ}
\textsc{J. Lindenstrauss, L. Tzafriri}, \emph{Classical Banach
spaces}, vol II, Springer-Verlag 1979.

\bibitem{LivTem14}
\textsc{E. Livshitz, V. Temlyakov}, 
Sparse approximation and recovery by greedy algorithms. IEEE Trans. Inform. Theory 60 (7) (2014), 3989--4000.

\bibitem{MT75}
\textsc{R. Maleev, S. Troyanski},  
On the moduli of convexity and smoothness in Orlicz spaces. Studia Math. 54 (2) (1975), 131--141.

\bibitem{PaWo22}
\textsc{A. Pawlewicz, M. Wojciechowski}, Marcinkiewicz sampling theorem for Orlicz spaces. Positivity 26 (3) (2022), Paper No. 56.

\bibitem{RR}
    \textsc{M.M. Rao and Z.D. Ren}, \emph{Theory of Orlicz spaces}.
    Monographs and Textbooks in Pure and Applied Mathematics, {\bf 146},
    Marcel Dekker Inc., 1991.

\bibitem{Sin1}
\textsc{I. Singer}. \emph{Bases in Banach spaces I}. Springer-Verlag, 1970.

\bibitem{Tem01}
\textsc{V.N. Temlyakov}, Greedy algorithms in Banach spaces. Adv. Comput. Math.  {\bf 14} (3)  (2001), 277--292. 

\bibitem{Tem11}
\textsc{V.N. Temlyakov}. \emph{Greedy Approximation}. Cambridge University Press, Cambridge, 2011.


\bibitem{Tem14}
\textsc{V.N. Temlyakov}, Sparse approximation and recovery by greedy algorithms in Banach spaces.
Forum Math, Sigma {\bf 2} (12) (2014), 26 p.

\bibitem{Tem15} 
\textsc{V.N. Temlyakov},
 Sparse Approximation with Bases, Advanced Courses in Mathematics CRM Barcelona, Birkh{\" a}user, Springer Basel 2015.


\bibitem{Tem18}
\textsc{V.N. Temlyakov}. \emph{Multivariate Approximation}. Cambridge University Press, Cambridge, 2018.

\bibitem{Zha11} \textsc{T. Zhang}, Sparse recovery with orthogonal matching pursuit under RIP. IEEE
Trans Inform Theory 57 (2011), 6215--6221.


\end{thebibliography}

\end{document}